\documentclass[11pt,reqno]{amsart}
 
\usepackage{graphicx,color}
\usepackage{amssymb,amscd, amsthm, latexsym, mathrsfs, epsfig}
\usepackage[all]{xy}
\usepackage{todonotes}

\newcommand\cy{\operatorname{CY}_3}
\newcommand\conf{\operatorname{C}}
\newcommand\PP{\mathbb P}
\newcommand\A{\mathcal{A}}
\newcommand\C{\mathbb C}

\newcommand\D{\mathcal{D}}
\newcommand\Q{\mathbb Q}
\newcommand\R{\mathbb R}
\newcommand\Z{\mathbb Z}

\renewcommand{\H}{\mathbb{H}}
\newcommand{\calH}{\mathcal{H}}
\newcommand{\LL}{\mathcal{L}}

\newcommand{\calS}{\mathcal{S}}

\newcommand{\U}{\mathcal{U}}
\newcommand{\V}{\mathcal{V}}

\newcommand\Aut{\operatorname{Aut}}
\newcommand\ad{\operatorname{ad}}
\newcommand\dt{\operatorname{DT}}
\newcommand\stab{\operatorname{Stab}}
\newcommand\bra{\langle}
\newcommand\ket{\rangle}
\newcommand{\pow}[1]{[\![ {#1} ]\!]}
\newcommand{\tg}[1]{T{#1}}

\newcommand{\del}{\partial}

\newcommand{\Hom}{\operatorname{Hom}}
\newcommand{\ext}{\operatorname{ext}}

\newcommand{\im}{\operatorname{im}}

\newcommand{\len}{\operatorname{len}}

\makeatletter \@addtoreset{equation}{section} \makeatother

\newtheorem{thmIntro}{Theorem}
\newtheorem{thm}{Theorem}[section]
\newtheorem{prop}[thm]{Proposition}
\newtheorem{lem}[thm]{Lemma}
\newtheorem{cor}[thm]{Corollary}

\theoremstyle{definition}
\newtheorem{definition}[thm]{Definition}
\newtheorem{exm}[thm]{Example} 
\newtheorem{rmk}[thm]{Remark}
 
\newcommand\narrowdots{\hbox to 1em{.\hss.\hss.}}

\title[]{A construction of Frobenius manifolds \\ from stability conditions}
\author{Anna Barbieri, Jacopo Stoppa and Tom Sutherland}
\date{}

\address{The University of Sheffield, Hicks Building, Hounsfield Road, Sheffield S3 7RH, United Kingdom}
\email{a.barbieri@sheffield.ac.uk}
\address{SISSA, via Bonomea 265, 34136 Trieste, Italy}
\email{jstoppa@sissa.it}
\address{Institut f\"ur Mathematik, Johannes Guthenberg-Universit\"at, Staudingerweg 9, 55128 Mainz, Germany}
\email{sutherland@uni-mainz.de}
\begin{document}

\begin{abstract} A finite quiver $Q$ without loops or $2$-cycles defines a $\cy$ triangulated category $\D(Q)$ and a finite heart $\A(Q) \subset \D(Q)$. We show that if $Q$ satisfies some (strong) conditions then the space of stability conditions $\stab(\A(Q))$ supported on this heart admits a natural family of semisimple Frobenius manifold structures, constructed using the invariants counting semistable objects in $\D(Q)$. In the case of $A_n$ evaluating the family at a special point we recover a branch of the Saito Frobenius structure of the $A_n$ singularity $y^2 = x^{n+1}$. 

We give examples where applying the construction to each mutation of $Q$ and evaluating the families at a special point yields a different branch of the maximal analytic continuation of the same semisimple Frobenius manifold. In particular we check that this holds in the case of $A_n$, $n \leq 5$.
\end{abstract}

\maketitle
\setcounter{tocdepth}{1}
\tableofcontents

\section{Introduction}
There is a strong formal analogy between wall-crossing structures for invariants enumerating semistable objects in abelian and triangulated categories and certain data attached to a semisimple Frobenius manifold. We refer to the wall-crossing theory developed in \cite{joy, ks, rein} and in particular to the structures described in \cite{bt_stab, bt_stokes, fgs, gmn, joy, jacopo}, as well as to the connections between spaces of stability conditions and semisimple Frobenius manifolds discovered by Bridgeland (see e.g. \cite{bridFrob, tomFrob}).  

The purpose of the present paper is to present an approach for turning this formal analogy into precise results, at least for some special examples.

We work with a well-studied class of $\cy$ triangulated categories $
\D = \D(Q, W)$ attached to quivers with potential $(Q, W)$, see e.g. \cite[Section 7]{bridsmith}. An important example is the $\cy$ category $\D(A_n) = \D(A_n, 0)$ attached to the Dynkin quiver $A_n$. Here we do not consider the difficult problem of identifying (a quotient of) the space of Bridgeland stability conditions $\stab(\D)$ globally with a manifold which is known a priori to admit a natural Frobenius structure (see e.g. \cite{tomFrob, bridsmith, sutherland}). Instead we first study in detail the different, local problem of constructing a germ of a semisimple Frobenius manifold using only the invariants enumerating semistable objects in $\D$. In the second part of the paper we work out how this germ changes under mutation of $Q$, and relate this to analytic continuation of the germ, in some special examples. So the structures we describe do not live on (a quotient of) $\stab(\D)$, but rather on the natural domain of definition of $n$-dimensional semisimple Frobenius structures, the configuration space of $n$ points in $\C$. More precisely we develop two main ideas.\\

\noindent $(A)$ Fixing a finite heart $\A \subset \D$, with $n$ isomorphism classes of simple objects, we study how to use the invariants enumerating semistable objects in $
\D$ to endow the space of stability conditions $\stab(\A) \cong \bar{\calH}^n$ supported on $\A$ with the structure of a semisimple Frobenius manifold. The \emph{canonical coordinates} $u_1, \ldots, u_n$ allow us to regard this semisimple Frobenius structure as living on an open subset of the configuration space 
\begin{equation*}
\conf_n(\C) = \{(u_1, \ldots, u_n) \in \C^n : i \neq j \Rightarrow u_i \neq u_j\} / \Sigma_n.
\end{equation*} 
Such a structure may then be continued analytically to the universal cover $\widetilde{\conf}_n(\C)$.\\

\noindent $(B)$ Suppose $\A' = \A'(Q', W') \subset \D$ is another heart, where $(Q', W')$ is obtained from $(Q, W)$ by \emph{quiver mutation}. Via the canonical coordinates $u'_i$ this gives a different semisimple Frobenius structure on another open subset of $\conf_n(\C)$. Since the categories $\D(Q)$, $\D(Q')$ are equivalent, it seems natural to ask if the two structures are related by analytic continuation, i.e. if they belong to the same semisimple Frobenius manifold structure on $\widetilde{\conf}_n(\C)$. This would give a way to understand the Frobenius manifold in terms of stability conditions (with different branches corresponding to tilts of the original heart).\\  
 
\noindent Under (very) restrictive conditions we can make the above picture precise.\\

\noindent $(A)$ If $\A \subset \D$ is a finite heart as above then by the results of \cite{AnnaJ} (based on \cite{bt_stab, ks} and especially on the work of Joyce \cite{joy}) there is a well-defined, canonical (but infinite-dimensional) \emph{Frobenius type structure} on an auxiliary bundle over $\stab(\A)$, defined over a ring of formal power series $\C\pow{\bf s}$, ${\bf s} = (s_1, \ldots, s_n)$, depending only on categorical data. There is a formal parameter $s_i$ for each class of a simple object. Joyce's original, formal construction in this case is recovered at the special point ${\bf s}_J = (s_1 = 1, \ldots, s_n = 1)$. The notion of a Frobenius type structure is due to Hertling and arises in the construction of Frobenius manifolds in singularity theory \cite{hert}. By general results of Hertling, a section $\zeta_{\A}$ of this auxiliary bundle can be used to pull back this structure to the tangent bundle $\tg{\stab(\A)}$, and one can give conditions under which this pullback is Frobenius. We say that the heart $\A \subset \D$ is good if there is a (good) section $\zeta_{\A}$ such that the pullback along $\zeta_{\A}$ gives a nontrivial \emph{jet} (quadratic or higher) of semisimple Frobenius manifold structures, i.e. if the pullback is a Frobenius structure modulo terms which lie in $({\bf s})^p$ for some $p \geq 3$. By general theory such jets can always be lifted to genuine families of semisimple Frobenius manifolds, which can then be evaluated at the geometrically meaningful Joyce point ${\bf s}_J$. (The lift is not unique, but for a fixed good heart $\A$ there is a \emph{natural} finite set of lifts, including a canonical \emph{minimal} choice). We give a characterisation of good sections $\zeta_{\A}$ (Corollary \ref{approxFrobType}), and work out the quadratic jets of the Stokes data (generalised monodromy) of the corresponding semisimple Frobenius manifolds (Lemma \ref{StokesLem}). Our whole construction is summarised in Theorem \ref{mainThm}.\\

We provide several examples, and we treat in detail the special case of $\D(A_n)$. (In this case the jet is of order $n$, and all the natural lifts coincide with the canonical one). In particular we prove the following result.

\begin{thmIntro}\label{mainThmAn}
Let $\A(A_n) \subset \D(A_n)$ be the standard heart. The construction described in (A) above (see Theorem \ref{mainThm} and Corollary \ref{AnConnection}) endows the space of stability conditions $\stab(\A(A_n))$ with the structure of a semisimple Frobenius manifold, and this coincides with a  branch of the Saito Frobenius structure on the unfolding space of the $A_n$ singularity $y^2 = x^{n+1}$.\\
\end{thmIntro}

\noindent $(B)$ We extend Theorem \ref{mainThmAn} to all mutations of $A_n$, and we give examples where all the quivers in a mutation class admit good sections, such that evaluating at ${\bf s}_J$ yields different branches of the same semisimple Frobenius manifold on $\conf_n(\C)$. Our examples currently include $A_n$ for $n \leq 5$ (Figure \ref{introA3figure}) as well as some other quivers obtained from triangulations of marked bordered surfaces (Figure \ref{introAnnulusFigure}). 

\begin{thmIntro}\label{mainThmMutations} Let $Q$ be a quiver mutation equivalent to $A_n$, $\A(Q) \subset \D(Q) \cong \D(A_n)$ the corresponding heart.
\begin{enumerate}  
\item[(i)] The construction described in $(A)$ above (see Theorem \ref{mainThm} and Lemma \ref{AnConnectionMutation}) endows $\stab(\A(Q))$ with the structure of a semisimple Frobenius manifold.

\item[(ii)] Suppose $n \leq 5$. Then this semisimple Frobenius manifold coincides with a branch of the maximal analytic continuation of the Saito Frobenius structure for the $A_n$ singularity $y^2 = x^{n+1}$.
\end{enumerate}
\end{thmIntro}

We expect that the statement (ii) can be generalised to all values of $n$. In the simplest case this gives a way to reconstruct the Saito Frobenius structure for the $A_2$ singularity $y^2 = x^3$ from the $\cy$ category $\D(A_2)$.  

Recall that analytic continuation of semisimple Frobenius manifolds can be understood in terms of a braid group action on Stokes matrices. We prove Theorem \ref{mainThmMutations} (ii) by writing down explicit braid relations for mutations of the basic quiver $A_n$. For simple mutations we observe a neat correspondence between mutations and braidings, but the general picture seems quite complicated (Cotti, Dubrovin and Guzzetti explained to us that a similar phenomenon occurs in their work on the quantum cohomology of Grassmannians \cite{cotti}).

\begin{figure}[ht]
\begin{equation*}
\begin{xy} 0;<.3pt,0pt>:<0pt,-.3pt>:: 
(0,0) *+{1} ="1",
(125,0) *+{2} ="2",
(243,0) *+{3} ="3",
"1", {\ar"2"},
"2", {\ar"3"},
\end{xy}
\quad\quad
\begin{xy} 0;<.3pt,0pt>:<0pt,-.3pt>:: 
(0,0) *+{1} ="1",
(126,123) *+{2} ="2",
(243,0) *+{3} ="3",
"2", {\ar"1"},
"1", {\ar"3"},
"3", {\ar"2"},
\end{xy}
\end{equation*}
\caption{Mutation equivalent quivers corresponding to hearts for $\D(A_3)$. They also correspond to different branches of the $A_3$ Frobenius manifold.}
\label{introA3figure}
\end{figure}
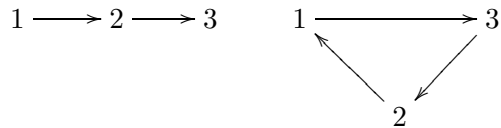
\begin{figure}[ht]
\begin{equation*}
\begin{xy} 0;<.3pt,0pt>:<0pt,-.3pt>:: 
(0,0) *+{1} ="1",
(261,0) *+{2} ="2",
(134,140) *+{3} ="3",
"1", {\ar"2"},
"1", {\ar"3"},
"2", {\ar"3"},
\end{xy}
\quad\quad
\begin{xy} 0;<.3pt,0pt>:<0pt,-.3pt>:: 
(0,0) *+{1} ="1",
(260,0) *+{2} ="2",
(134,140) *+{3} ="3",
"2", {\ar"1"},
"1", {\ar|*+{\scriptstyle 2}"3"},
"3", {\ar"2"},
\end{xy}
\end{equation*}
\caption{Mutation equivalent triangulation quivers for the annulus. They also correspond to different branches of the same semisimple Frobenius manifold.}
\label{introAnnulusFigure}
\end{figure}
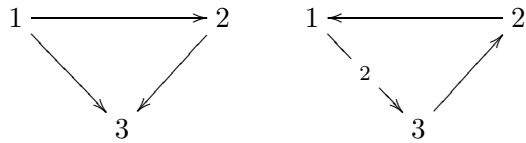

\medskip

\noindent\textbf{Plan of the paper.} Section \ref{backgroundSection} contains background material about the infinite-dimensional Frobenius type structure on $\stab(\A)$. Section \ref{projectionSec} uses this to construct jets of families of finite-dimensional, semisimple Frobenius type structures, and Section \ref{liftSection} discusses how to lift these jets naturally to genuine families. Both sections contain some explicit examples. Section \ref{pullbackSection} recalls the notion of a Frobenius manifold, due to Dubrovin \cite{dubrovin}, and then uses Hertling's pullback to turn our structures into families of semisimple Frobenius manifolds. There are further explicit examples in Section \ref{examplesSec}. Section \ref{AnSection} contains the proof of Theorem \ref{mainThmAn}. Section \ref{mutationsSec} discusses the relation between mutations and analytic continuation and contains the proof of Theorem \ref{mainThmMutations}. 

\medskip

\noindent\textbf{Acknowledgements.} We would like to thank Tom Bridgeland, Sara Angela Filippini, Mario Garcia-Fernandez and Giordano Cotti. We are grateful to the anonymous Referee for a careful reading of the manuscript.

The research leading to these results has received funding from the European Research Council under the European Union's Seventh Framework Programme (FP7/2007-2013) / ERC Grant agreement no. 307119.

\section{Infinite-dimensional Frobenius type structures}\label{backgroundSection}

We start by briefly recalling the categorical setup described e.g.\ in \cite[Section 7]{bridsmith}. Let $\D$ be a $\C$-linear triangulated category of finite type. Following \cite{bridsmith} in this paper we always assume that
\begin{enumerate}
\item[$\bullet$] $\D$ admits a bounded t-structure whose heart $\A \subset \D$ is a finite abelian category, with $n$ distinct simple objects up to isomorphism; 
\item[$\bullet$] $\D$ is $\cy$, that is for all $A, B \in \D$ there are functorial isomorphisms
\begin{equation*}
\Hom_{\D}(A, B[i]) = \Hom_{\D}(B, A[3-i])^*.
\end{equation*}
\end{enumerate}
The finiteness condition is especially restrictive. It implies that the Grothe\-dieck group $K(\D)$ (generated by isomorphism classes of objects modulo the relations given by exact triangles) is isomorphic to $\Z^n$. 
\begin{definition} The Euler form on $K(\D)$ is defined by
\begin{equation*}
\bra E, F\ket = \sum_{i \in \Z} (-1)^i \hom_{\D}(E, F[i]). 
\end{equation*}
By the $\cy$ condition this is a skew-symmetric bilinear form on $K(\D)$.
\end{definition}
One can associate a quiver to a category $\D$ as above with a fixed finite heart $\A\subset \D$. 
\begin{definition} Let $\A \subset \D$ be a finite heart. We define a quiver $Q(\A)$ given by
\begin{enumerate}
\item[$\bullet$] the set of vertices is the set of isomorphism classes of simples $S_i \in \A$; 
\item[$\bullet$] there are $n_{ij} = \ext^1_{\A}(S_i, S_j)$ arrows between the vertices $i$ and $j$.
\end{enumerate} 
\end{definition}
The main point is that, fixing a choice of potential, one can essentially reverse this construction.
\begin{thm}[{\cite[Theorem 7.2]{bridsmith}}] Let $(Q, W)$ be a quiver with reduced potential. Then one can construct a $\C$-linear, finite type $\cy$ category $\D(Q, W)$, and a finite heart $\A = \A(Q, W) \subset \D(Q, W)$, such that the associated quiver $Q(\A)$ is isomorphic to $Q$. In particular the simple objects of $\A$ are in natural bijection with vertices of $Q$, and the arrows of $Q$ give bases for the extension spaces between them.
\end{thm}
In the rest of this paper we write $\A$ for a category $\A(Q, W)$, with its natural embedding 
\begin{equation*}
\A = \A(Q, W) \subset \D(Q, W) = \D. 
\end{equation*}
In particular the natural inclusion of Grothendieck groups $K(\A) \hookrightarrow K(\D)$ is an isomorphism. So $K(\A)$ is the lattice generated by the classes of simples $[S_i], i = 1, \ldots, n$. 

\begin{definition} The effective cone $K(\A)_{> 0} \subset K(\A)$ is the submonoid generated by the classes of simples $[S_i], i = 1, \ldots, n$. Equivalently it is the submonoid of $K(\A)$ generated by classes of nonzero objects. 
\end{definition}

Let us denote by $\bar{\calH}$ the semi-closed upper half plane $\H \cup \R_{< 0}$.

\begin{definition} The space of stability conditions $\stab(\A)$ is the open subset 
\begin{equation*}
\{ Z \in \Hom(K(\A), \C) \, : \, Z(K_{>0}(\A)) \subset \bar{\calH} \} \subset \Hom(K(\A), \C).
\end{equation*}
\end{definition}

Thus $\stab(\A)$ is a complex manifold, biholomorphic to $\bar{\calH}^n$. We refer to points $Z \in \stab(\A)$ as central charges. A stability condition for $\D$ is determined by a pair $(\A,Z)$, consisting of a heart $\A$ and a central charge $Z$ (see e.g. \cite[Section 7.5]{bridsmith}). 

There is a well-defined enumerative theory for semistable objects of a given class in $\D$, \cite[Chapter 7]{js}.
\begin{definition} We denote by $\dt_{\A}(\alpha, Z) \in \Q$ the Donaldson-Thomas type invariant virtually enumerating objects in $\D$ of class $\alpha \in K(\D) \cong K(\A)$ which are semistable with respect to the stability condition determined by the pair $(\A, Z)$.  
\end{definition}

The shift functor $[1] \in \Aut(\D)$ preserves the class of semistable objects and acts on $K(\D)$ as $-I$, so we have
\begin{equation*}
\dt_{\A}(\alpha, Z) = \dt_{\A}(-\alpha, Z).
\end{equation*} 

\medskip

We turn to describing the natural infinite-dimensional Frobenius type structure on the space of stability conditions $\stab(\A)$. We need some preliminary notions.

\begin{definition} We denote by $\C[K(\A)]$ the twisted group-algebra on $K(\A)$. It is generated by $x_{\alpha}, \alpha \in K(\A)$, with commutative product 
\begin{equation*}
x_{\alpha} x_{\beta} = (-1)^{\bra \alpha, \beta \ket}x_{\alpha + \beta}.
\end{equation*}
It is a Poisson Lie algebra with the bracket
\begin{equation*}
[x_{\alpha}, x_{\beta}] = (-1)^{\bra \alpha, \beta \ket} \bra \alpha, \beta\ket x_{\alpha + \beta}.
\end{equation*}
\end{definition}

The classes $[S_i] \in K(\A)_{> 0}$ give a canonical basis of $K(\A)$ with respect to which we decompose every other class, 
\begin{equation*}
\alpha = \sum^n_{i = 1} a_i [S_i]. 
\end{equation*}  
We write 
\begin{equation*}
[\alpha]_{\pm} = \sum^n_{i = 1} [a_i]_{\pm} [S_i]
\end{equation*}
where $[a_i]_{\pm}$ denote the positive and negative parts, i.e.\ $[a_i]_{+} = \max(a_i, 0)$, $[a_i]_{-} = \min(a_i, 0)$, and we define
\begin{equation*}
\len(\alpha) = \sum^n_{i = 1} | a _i |. 
\end{equation*}
Let $s_1, \ldots, s_n$ denote formal parameters. We denote this collection of formal parameters by ${\bf s}$, and set 
\begin{equation*}
{\bf s}^{\alpha} = \prod^n_{i = 1} s^{a_i}_i.
  \end{equation*}
In particular ${\bf s}^{[\alpha]_{+} - [\alpha]_{-}}$ is a monomial, not just a Laurent monomial. We denote the maximal ideal and its powers by
\begin{equation*}
({\bf s})^{N} = (s_1, \ldots, s_n)^N \subset \C[s_1, \ldots, s_n].
\end{equation*}

Let us introduce the coefficients 
\begin{equation*} 
c(\alpha_1, \cdots, \alpha_k) = \sum_T \frac{1}{2^{k - 1}}\prod_{\{i \to j\} \subset T} (-1)^{\bra \alpha_i, \alpha_j\ket} \bra \alpha_i, \alpha_j\ket.
\end{equation*}
where the sum is over all connected trees $T$ with vertices labelled by $\{1, \narrowdots, k\}$, endowed with an orientation compatible with the labelling.
\begin{definition} For $\alpha \neq 0$ we introduce the complex-valued formal power series in $\bf{ s}$ given by  
\begin{align}\label{fFunFps}
\nonumber  f ^{\alpha}_{\bf s}(Z) = &\sum_{\alpha_1 + \cdots + \alpha_k = \alpha,\, Z(\alpha_i) \neq 0} c(\alpha_1, \ldots, \alpha_k) J_k(Z(\alpha_1), \ldots , Z(\alpha_k))\\ &\prod^k_{i = 1} {\bf s}^{[\alpha_i]_+ - [\alpha_i]_-} \dt_{\A}(\alpha_i, Z),
\end{align}
where $J_k\!: (\C^*)^k \to \C$ are the sectionally holomorphic functions introduced by Joyce \cite{joy}. Note that this is well-defined because there are only finitely many decompositions in \eqref{fFunFps} modulo $({\bf s})^N$ for $N \gg 1$. The holomorphic generating function of $\dt_{\A}$ is the holomorphic function with values in $\C[K(\A)]\pow{\bf s}$ given by
\begin{equation*}
\Phi_{\bf s}(Z) = \sum_{\alpha \neq 0}  f ^{\alpha}_{\bf s}(Z) x_{\alpha}. 
\end{equation*}
\end{definition}
\begin{rmk} The original, formal definition of a Joyce holomorphic generating function is recovered at the special point 
\begin{equation*}
{\bf s}_J = (s_1 = 1, \ldots, s_n = 1).
\end{equation*}
We do not know how to give a meaning to this specialisation in general. In the following we will reduce to a finite-dimensional context and specialise to the point ${\bf s}_J$ after the reduction.
\end{rmk}
\begin{prop}[{\cite[Proposition 3.17]{AnnaJ}}] The coefficients of the formal power series $ f ^{\alpha}_{\bf s}(Z)$ are holomorphic functions of $Z \in \stab(\A)$. 
\end{prop}

Let us also recall the notions of Frobenius structure on the fibres of an arbitrary holomorphic vector bundle.

\begin{definition}[Frobenius type structure, {\cite[Definition 5.6]{hert}}] Let $K \to M$ be a holomorphic vector bundle on a complex manifold. A Frobenius type structure on $K$ is given by a collection $(\nabla^r, C, \U, \V, g)$ of \emph{holomorphic} objects with values in $K$, where
\begin{enumerate}
\item[$\bullet$] $\nabla^r$ is a flat connection, 
\item[$\bullet$] $C$ is a Higgs field, i.e.\ a $1$-form with values in endomorphisms such that $C \wedge C = 0$,
\item[$\bullet$] $\U, \V$ are endomorphisms,
\item[$\bullet$] $g$ is a quadratic form,
\end{enumerate} 
such that the conditions
\begin{align}\label{FrobTypeCond1}
\nonumber \nabla^r(C) &=0,\\
\nonumber [C, \U] &= 0,\\
\nonumber \nabla^r(\V) &= 0,\\
\nabla^r(\U) - [C, \V] + C &= 0
\end{align}
hold. Moreover we require that $g$ is covariant constant with respect to $\nabla^r$, and that $C$, $\U$ are symmetric and $\V$ is skew-symmetric with respect to $g$. 
\end{definition}
The function $\dt_{\A}$ determines a Frobenius type structure on a bundle over $\stab(\A)$. The endomorphism $\U$ and Higgs field $C$ are given essentially by the central charge $Z$ and its exterior differential. The endomorphism $\V$ and flat connection $\nabla^r$ are given essentially by the adjoint action of the holomorphic generating function.
\begin{prop}[{\cite[Proposition 3.17]{AnnaJ}}]\label{DTFrobTypeStr} Let $K \to \stab(\A)$ denote the trivial vector bundle with fibre $\C[K(\A)]\pow{\bf s}$. Consider the following (formal power series) holomorphic objects with values in $K$: 
\begin{enumerate}
\item[$\bullet$] a connection   
\begin{equation*}
\nabla^r_{{\bf s}} = d + \sum_{\alpha \neq 0} \ad f^{\alpha}_{\bf s}( Z) x_{\alpha} \frac{d Z(\alpha)}{Z(\alpha)},
\end{equation*}
\item[$\bullet$] a $1$-form with values in endomorphisms
\begin{equation*}
C = - d Z,
\end{equation*}
(acting as $C_X(x_{\alpha}) = X Z(\alpha) x_{\alpha}$ for all holomorphic vector fields $X$)
\item[$\bullet$] endomorphisms
\begin{equation*}
\U = Z
\end{equation*}
(acting as $Z(x_{\alpha}) = Z(\alpha) x_{\alpha}$) and
\begin{equation*}
\V_{\bf s} = \ad \Phi_{\bf s}(Z) = \ad \sum_{\alpha \neq 0} f^{\alpha}_{\bf s}(Z) x_{\alpha}
\end{equation*}
\item[$\bullet$] a quadratic form 
\begin{equation*}
g(x_{\alpha}, x_{\beta}) = \delta_{\alpha \beta}.
\end{equation*}
\end{enumerate} 
Then $(\nabla^r_{\bf s}, C, \U, \V_{\bf s}, g)$ defines a $\C\pow{\bf s}$-linear Frobenius type structure. 
\end{prop}
Note that here we use the Lie algebra structure on $\C[K(\A)]$ just to describe endomorphisms of $K$, i.e.\ we work with a vector bundle not a principal bundle. 

Although we do not reproduce the proof of the Proposition here, we should point out that it follows quite easily from Joyce's work \cite{joy}. In particular we note that 
\begin{enumerate}
\item[$\bullet$] flatness of $\nabla^r_{\bf s}$ follows from a partial differential equation satisfied by $\Phi_{\bf s}(Z)$, which in turn follows easily from results in \cite{joy};
\item[$\bullet$] the covariant constancy of $g$ with respect to $\nabla^r_{\bf s}$ uses the triangulated and $\cy$ properties of $\D$ in an essential way. 
\end{enumerate}
What is needed for the latter condition is the equality
\begin{equation*}
f^{\alpha - \beta}_{\bf s}(Z) \bra \alpha, \beta \ket = - f^{\beta - \alpha}_{\bf s}(Z) \bra \beta, \alpha \ket  
\end{equation*}
for all $\alpha, \beta \in K(\D)$. In our case this follows from 
\begin{equation*}
f^{\alpha - \beta}_{\bf s}(Z) = f^{\beta - \alpha}_{\bf s}(Z)
\end{equation*}
which holds since $[1] \in \Aut(\D)$ preserves semistability, and
\begin{equation*}
\bra \alpha, \beta \ket = - \bra \beta, \alpha\ket
\end{equation*}
which follows from the $\cy$ property.\\

The general notion of a bundle with a Frobenius type structure is of course motivated by the special case of a Frobenius manifold $M$, corresponding to certain structures on the tangent bundle $\tg{M}$. The notion of a Frobenius manifold and its relation to a Frobenius type structure on $\tg{M}$ are briefly recalled at the beginning of Section \ref{pullbackSection} (see Definition \ref{FrobDefinition} and Lemma \ref{HertLem}).

\section{Approximate finite-dimensional Frobenius type structures}\label{projectionSec}   
 Let $K \to M$ be a holomorphic vector bundle with Frobenius type structure $(\nabla^r, C, \U, \V, g)$ on a complex manifold $M$. Write $\tg{M}$ for the holomorphic tangent bundle to $M$. A holomorphic section $\zeta \in H^0(M, K)$ can be contracted with the Higgs field $C$ to give a map
\begin{equation}\label{HiggsDerivative}
- C_{\bullet}(\zeta)\!: \tg{M} \to K
\end{equation} 
i.e.\ minus the derivative of the section $\zeta$ along the Higgs field.

Our constructions in the present Section are motivated by the following result.

\begin{thm}[Hertling {\cite[Theorem 5.12]{hert}}]\label{HertThm} Suppose that $\zeta$ is a global section of $K$ such that 
\begin{enumerate}
\item[$\bullet$] it is a flat section with respect to the flat connection of the Frobenius type structure, $\nabla^r(\zeta) = 0$,
\item[$\bullet$] it is homogeneous with respect to the endomorphism $\V$, i.e.\ we have $\V(\zeta) = \frac{d}{2} \zeta$ for some $d \in \C$,
\item[$\bullet$] the map \eqref{HiggsDerivative} is an isomorphism. 
\end{enumerate}
Then the pullback of $(\nabla^r, C, \U, \V, g)$ along the map \eqref{HiggsDerivative} gives a Frobenius manifold structure on $M$ with unit field given by the pullback of the section $\zeta$ and with conformal dimension $2 - d$ (see Definition \ref{FrobDefinition}). 
\end{thm}
We would like to apply Theorem \ref{HertThm} to the Frobenius type structure on the bundle $K \to U$ given by Proposition \ref{DTFrobTypeStr}, where $U \subset \stab(\A)$ is a suitable open subset. Note that for a fixed holomorphic section $\zeta$ of $K$ the map $- C_{\bullet}(\zeta) = d Z(\zeta)$ is onto the finite rank subbundle defined by
\begin{equation*}
K(\zeta) = \im (d Z (\zeta)) \subset K.
\end{equation*}
It is natural to ask when $\zeta$ is in fact a section of the bundle $K(\zeta)$.
\begin{lem}\label{zetaLem} We have that $\zeta$ is a section of $K(\zeta)$ if and only if there are elements $\alpha_1, \ldots, \alpha_r \in K(\A)$, linearly independent over $\R$, such that
\begin{equation*}
\zeta = \sum^{r}_{i = 1} c_{i}(Z, {\bf s}) x_{\alpha_i} + \sum_{\substack{a_1 + \cdots + a_r = 1\\ a_1\alpha_1 + \cdots + a_r \alpha_r \neq \alpha_i, \, i = 1, \dots, r}} c_{a_1, \cdots, a_r} (Z, {\bf s}) x_{a_1\alpha_1 + \cdots + a_r \alpha_r}
\end{equation*} 
 where $c_{i}(Z, {\bf s})$, $c_{a_1, \cdots, a_n} (Z, {\bf s})$ are formal power series in the variables ${\bf s}$ with holomorphic coefficients.
\end{lem} 
\begin{proof} A holomorphic section $\zeta$ of $K$ takes the form 
\begin{equation*}
\zeta = \sum_{\alpha \in I'} c_{\alpha}(Z, {\bf s}) x_{\alpha}
\end{equation*}
where $I' \subset K(\A)$ and the $c_{\alpha}(Z, {\bf s})$, $\alpha \in I'$ are formal power series which do not vanish identically. Let $X$ be a holomorphic vector field. We compute
\begin{align*}
dZ(X)(\zeta)  &= \sum_{\alpha \in I'} c_{\alpha}(Z, {\bf s}) dZ(X)(x_{\alpha})\\
&= \sum_{\alpha \in I'} c_{\alpha}(Z, {\bf s}) X(Z(\alpha)) x_{\alpha}. 
\end{align*}
So $\zeta \in K(\zeta)$ if and only if there exists a holomorphic vector field $X$ such that for all $\alpha \in I'$ we have
\begin{equation*}
X(Z(\alpha)) = 1.
\end{equation*}
Choose a maximal set of elements $\alpha_1, \ldots, \alpha_r$ of $I'$ which are linearly independent over $\R$. The functions $Z(\alpha_1), \ldots, Z(\alpha_r)$ are part of a local coordinate system $u_1, \ldots, u_n$ on $\stab(\A)$ with $u_i = Z(\alpha_i)$ for $i = 1, \ldots, r$. The general solution $X$ to $X u_i = 1$, $i = 1, \ldots, r$ is a vector field 
\begin{equation*}
X = \sum^r_{i=1}\del_{u_i} + \sum^{n}_{j = r+1} b_{j} \del_{u_j}
\end{equation*}
for arbitrary $b_j$. All the other $\alpha \in I = I' \setminus \{\alpha_1, \ldots, \alpha_r\}$ are linear combinations $\alpha = \sum^r_{i = 1} a_i \alpha_i$. The condition $X(Z(\alpha)) = 1$ holds for $\alpha \in I$ if and only if $\sum^r_{i = 1} a_i = 1$.
\end{proof}
The following result is clear from the proof of Lemma \ref{zetaLem}.
\begin{cor}\label{zetaCor} Suppose $\zeta$ is a section of $K(\zeta)$. The map $- C_{\bullet}(\zeta) = dZ(\zeta)\!: T U \to K(\zeta)$ is injective (and so an isomorphism) if and only if in Lemma \ref{zetaLem} we have $r = n$ and the functions $c_1(Z, {\bf s}), \ldots, c_n(Z, {\bf s})$ are nowhere vanishing on $U$. In this case $K(\zeta) \subset K$ is the subbundle generated by
\begin{align*} 
d Z(\del_{Z(\alpha_i)})(\zeta) = x_{\alpha_i} +  \sum_{\substack{a_1 + \cdots + a_r = 1\\ a_1\alpha_1 + \cdots + a_r \alpha_r \neq \alpha_j, \, j = 1, \dots, n}} a_i c_{a_1, \cdots, a_n} (Z, {\bf s}) x_{a_1\alpha_1 + \cdots + a_n \alpha_n}
\end{align*}
for $i =  1, \ldots, n$ (following the notation of Lemma \ref{zetaLem}). 
\end{cor}
We will always assume that we are in the situation of Corollary \ref{zetaCor}. Thus $\zeta$ is a section of the bundle $K(\zeta)$ and the natural map $-C_{\bullet}(\zeta)\!: T U \to K(\zeta)$ is an isomorphism, so we can contemplate applying Theorem \ref{HertThm}. 

\begin{definition} Let us denote by $\pi^{\zeta}\!: K \to K(\zeta)$ the orthogonal projection onto $K(\zeta)$ with respect to the quadratic form $g$ of Proposition \ref{DTFrobTypeStr}. Note that $K$ is infinite-dimensional, but we will only apply $\pi^{\zeta}$ to sections of $K$ for which it is well-defined. We write  $\nabla^{r, \zeta}_{\bf s}, C^{\zeta}, \U^{\zeta}, \V^{\zeta}_{\bf s}, g^{\zeta}$ for the connection, endomorphisms and quadratic form given respectively by
\begin{equation*}
\pi^{\zeta} \circ \nabla^{r}_{\bf s}|_{K(\zeta)}, \pi^{\zeta} \circ C|_{K(\zeta)}, \pi^{\zeta} \circ \U|_{K(\zeta)}, \pi^{\zeta} \circ\V_{\bf s}|_{K(\zeta)}, g|_{K(\zeta)}.
\end{equation*}
\end{definition}
The holomorphic data
\begin{equation*}
(K(\zeta), \nabla^{r, \zeta}_{\bf s}, C^{\zeta}, \U^{\zeta}, \V^{\zeta}_{\bf s}, g^{\zeta})
\end{equation*}
give a formal family of structures on $K(\zeta)$, parametrised by ${\bf s}$. 

This is not in general a family of Frobenius type structures, and moreover $\zeta$ is not in general a flat or $\V_{\bf s}$-homogeneous section of $K(\zeta)$: the Frobenius type, $\zeta$-flatness and conformal conditions do not hold modulo terms in ${\bf s}$ of arbitrarily high degree.

However one can still ask whether this formal family \emph{osculates} a family of Frobenius type structures on $K(\zeta)$ to some order. More precisely we ask whether the Frobenius type, $\zeta$-flatness and conformal conditions hold modulo some power $({\bf s})^p$ of the ideal $({\bf s}) = (s_1, \cdots, s_n)$ with $p \geq 3$, i.e. modulo terms which are at least cubic. 

In the rest of this section we study this problem. Even if it makes sense more generally we will restrict to the case when the bundle $K(\zeta)$ is preserved by the Higgs field and the endomorphism $\U$. This condition is clarified by the following result. 
 
\begin{lem}\label{HiggsInvariance} Let $\zeta$ be a holomorphic section of $K$ (we do not assume a priori that $\zeta$ is a section of $K(\zeta)$). The following are equivalent:
\begin{enumerate}
\item[$\bullet$] $K(\zeta)$ is preserved by the Higgs field $C = - d Z$,
\item[$\bullet$] $K(\zeta)$ is preserved by the endomorphism $\U = Z$,
\item[$\bullet$] the section $\zeta$ has the form 
\begin{equation*}
\zeta = \sum^r_{i = 1} c_i (Z, {\bf s }) x_{\alpha_i}
\end{equation*}
where $\alpha_1, \ldots, \alpha_r \in K(\A)$ are linearly independent over $\R$.    
\end{enumerate}
\end{lem}

\begin{proof} Suppose $K(\zeta)$ is preserved by $C$. Let us write
\begin{equation*}
\zeta = \sum_{\alpha \in I} c_{\alpha}(Z, {\bf s}) x_{\alpha}
\end{equation*}
where $I \subset K(\A)$ and the $c_{\alpha}(Z, {\bf s})$, $\alpha \in I$ are formal power series which do not vanish identically. Then by construction sections of the bundle $K(\zeta)$ have the form
\begin{equation*} 
d Z(X)(\zeta) = \sum_{\alpha \in I} c_{\alpha}(Z, {\bf s}) X(Z(\alpha)) x_{\alpha}
\end{equation*}
as $X$ varies in the space of holomorphic vector fields on $U$. In order to simplify the notation we set $\zeta_X = d Z(X)(\zeta)$. Acting with the Higgs field $C = -dZ$ contracted with a holomorphic field $Y$ we find  
\begin{equation*}
C_Y \zeta_X = - \sum_{\alpha \in I} c_{\alpha}(Z, {\bf s}) X(Z(\alpha)) Y(Z(\alpha)) x_{\alpha}.
\end{equation*}
So $C_Y \zeta_X$ is a section of $K(\zeta)$ if and only if there exists a holomorphic field $W = W(X, Y)$ such that for all $\alpha \in I$ we have
\begin{equation}\label{quadraticVectorField}
W(X, Y)(Z(\alpha)) = - X(Z(\alpha)) Y(Z(\alpha)).
\end{equation} 
Let $\alpha_1, \ldots, \alpha_r$ denote a maximal set of $\R$-linearly independent elements of $I$. Suppose there is a nontrivial $\alpha \in I \setminus \{\alpha_1, \ldots, \alpha_r\}$. Decomposing $\alpha = a_1 \alpha_1 + \cdots + a_r \alpha_r$ we find
\begin{align}\label{quadratic1}
\nonumber W(X, Y)(Z(\alpha)) &= \sum^r_{i = 1}a_i W(X, Y)(Z(\alpha_i))\\
&= - \sum^r_{i = 1}a_i X(Z(\alpha_i)) Y(Z(\alpha_i))
\end{align}
where the second equality follows from applying \eqref{quadraticVectorField} to each $\alpha_i$. On the other hand applying  \eqref{quadraticVectorField} to $\alpha$ gives
\begin{equation}\label{quadratic2}
W(X, Y)(Z(\alpha)) = - \sum^r_{i, j = 1} a_i a_j X(Z(\alpha_i)) Y(Z(\alpha_j)).
\end{equation}
By \eqref{quadratic1} for all $k \neq l$ we have
\begin{equation*}
W(\del_{Z(\alpha_k)}, \del_{Z(\alpha_l)})(Z(\alpha)) = 0.
\end{equation*}
On the other hand \eqref{quadratic2} gives for all $k \neq l$
\begin{equation*}
W(\del_{Z(\alpha_k)}, \del_{Z(\alpha_l)})(Z(\alpha)) = - a_k a_l.
\end{equation*}
It follows that $a_k$ or $a_l$ vanish for all $k \neq l$, i.e. $\alpha$ must be a multiple of one of $\alpha_1, \ldots, \alpha_r$. By \eqref{quadratic1} for all $k$ we have
\begin{equation*}
W(\del_{Z(\alpha_k)}, \del_{Z(\alpha_k)})(Z(\alpha)) = -a_k.
\end{equation*}
On the other hand \eqref{quadratic2} gives for all $k$
\begin{equation*}
W(\del_{Z(\alpha_k)}, \del_{Z(\alpha_k)})(Z(\alpha)) = - a^2_k.
\end{equation*}
It follows that we must have $a_k = 0$ or $a_k = 1$ for all $k$. Since we already know that at most one $a_k$ does not vanish we see that $\alpha$ must be one of $\alpha_1, \ldots, \alpha_r$, a contradiction. Thus the section $\zeta$ must take the form
\begin{equation*}
\zeta = \sum^r_{i = 1} c_i (Z, {\bf s }) x_{\alpha_i}
\end{equation*}
where $\alpha_1, \ldots, \alpha_r \in K(\A)$ are linearly independent over $\R$. 

Conversely a straightforward computation shows that for a section $\zeta$ of this form and arbitrary fields $X, Y$ we can find a vector field $W(X, Y)$ as above, so $K(\zeta)$ is preserved by $C$. 

The argument for the endomorphism $\U$ is almost identical and we leave it to the reader.
\end{proof}
In the following we always assume that the bundle $K(\zeta)$ is preserved by the Higgs field $C$ and the endomorphism $\U$, and that the map $-C_{\bullet}(\zeta)\!: T U \to K(\zeta)$ is an isomorphism. According to Corollary \ref{zetaCor} and Lemma \ref{HiggsInvariance} this holds precisely when the section $\zeta$ takes the form
\begin{equation}\label{specialZeta}
\zeta = \sum^n_{i = 1} c_i (Z, {\bf s }) x_{\alpha_i}
\end{equation}
where $\alpha_1, \ldots, \alpha_n \in K(\A)$ are a basis over $\R$ and the functions $c_i (Z, {\bf s })$ are nowhere vanishing on $U$. Our family of structures on $K(\zeta)$ is then given by 
\begin{equation*}
(\nabla^{r, \zeta}_{\bf s}, C|_{K(\zeta)}, \U|_{K(\zeta)}, \V^{\zeta}_{\bf s}, g|_{K(\zeta)}).
\end{equation*}
\begin{lem}\label{nablarzetaFlat} Pick a section $\zeta$ of the form \eqref{specialZeta} (so $\zeta$ is a section of $K(\zeta)$ and the latter is preserved by $C$ and $\U$). Fix $i, j = 1, \ldots n$, and $p \geq 3$. Suppose one of the following alternatives holds: we have
\begin{equation*}
\len(\alpha_j - \alpha_i) \geq p,
\end{equation*}
or
\begin{enumerate}
\item[]\emph{quadratic condition:} for all $k \neq i, j$ we have
\begin{equation}\label{quadraticCondition}
\bra \alpha_j, \alpha_i\ket\bra \alpha_j - \alpha_k, \alpha_k - \alpha_i\ket = \bra \alpha_j, \alpha_k\ket \bra \alpha_k, \alpha_i\ket,
\end{equation}

\item[]\emph{vanishing condition:} for all nontrivial decompositions $\alpha_j - \alpha_i = \beta + \gamma$ with $\beta, \gamma$ not equal to $\alpha_j - \alpha_k$, $\alpha_k - \alpha_i$ the product 
\begin{equation}\label{vanishingCondition}
\bra \beta, \gamma\ket  f ^{\beta}_{\bf s}(Z)  f ^{\gamma}_{\bf s}(Z)
\end{equation}
lies in $({\bf s})^p$. 
\end{enumerate}
\noindent Then the curvature component $g(x_{\alpha_j}, F(\nabla^{r, \zeta}_{\bf s}) x_{\alpha_i})$ vanishes modulo terms which are of order at least $p$ in $\bf s$.
\end{lem}
\begin{proof} Under the assumptions the bundle $K(\zeta)$ is the subbundle generated by the sections $x_{\alpha_1}, \ldots, x_{\alpha_n}$. Let us write the connection $\nabla^{r, \zeta}_{\bf s}$ with respect to this local trivialisation. We compute
\begin{align*}
\nabla^{r, \zeta}_{\bf s}(x_{\alpha_i}) &=  \sum_{\alpha \neq 0} \pi^{\zeta}\Big( f ^{\alpha}_{\bf s}(Z) (-1)^{\bra \alpha, \alpha_i\ket} \bra \alpha, \alpha_i\ket x_{\alpha + \alpha_i} \Big)d\log Z(\alpha)\\
&= \sum^{n}_{j = 1} (-1)^{\bra \alpha_j, \alpha_i\ket} \bra \alpha_j, \alpha_i\ket  f ^{\alpha_j - \alpha_i}_{\bf s}(Z) x_{\alpha_j} d\log Z(\alpha_j - \alpha_i).  
\end{align*}
So the connection matrix of $1$-forms $A$ in this local trivialisation is given by 
\begin{equation}
A_{j i} = (-1)^{\bra \alpha_j, \alpha_i\ket} \bra \alpha_j, \alpha_i\ket  f ^{\alpha_j - \alpha_i}_{\bf s}(Z) d\log Z(\alpha_j - \alpha_i)    \label{Aij}
\end{equation} 
and the curvature form $dA + A \wedge A$ is the matrix of $1$-forms 
\begin{align*}
&(-1)^{\bra \alpha_j, \alpha_i\ket} \bra \alpha_j, \alpha_i\ket d f ^{\alpha_j - \alpha_i}_{\bf s}(Z) \wedge d\log Z(\alpha_j - \alpha_i)\\ 
&+ \sum^n_{k = 1} (-1)^{\bra \alpha_j, \alpha_k\ket + \bra \alpha_k, \alpha_i\ket} \bra \alpha_j, \alpha_k\ket \bra \alpha_k, \alpha_i\ket f ^{\alpha_j - \alpha_k}_{\bf s}(Z)  f ^{\alpha_k - \alpha_i}_{\bf s}(Z)\\&d\log Z(\alpha_j - \alpha_k) \wedge d\log Z(\alpha_k - \alpha_i). 
\end{align*}
We see from this expression that if $\len(\alpha_j - \alpha_i) \geq p$ then the component $(dA  + A \wedge A)_{ji}$ vanishes modulo $({\bf s})^p$. In general, flatness of the connection $\nabla^r_{\bf s}$ on $K$ is expressed by the Joyce PDE \cite{joy}
\begin{equation*}
d  f ^{\alpha}_{\bf s}(Z) = - \sum_{\alpha = \beta + \gamma} (-1)^{\bra \beta, \gamma\ket} \bra \beta, \gamma\ket  f ^{\beta}_{\bf s}(Z)  f ^{\gamma}_{\bf s}(Z) d\log Z(\beta)
\end{equation*}
for all $\alpha \neq 0$ (summing over decompositions with $\beta, \gamma \neq 0$). In our case we choose $\alpha = \alpha_j - \alpha_i$ and write the PDE in the form
\begin{align*}
d f ^{\alpha_j - \alpha_i}_{\bf s}(Z) &= - \sum_{k \neq i, j} (-1)^{\bra \alpha_j - \alpha_k, \alpha_k - \alpha_i\ket} \bra \alpha_j - \alpha_k, \alpha_k - \alpha_i\ket  f ^{\alpha_j - \alpha_k}_{\bf s}(Z) f ^{\alpha_k - \alpha_i}_{\bf s}(Z)\\
&\Big( d\log Z(\alpha_j - \alpha_k) - d\log Z(\alpha_k - \alpha_i)\Big)\\
&- \sum_{\alpha_j - \alpha_i = \beta' + \gamma'} (-1)^{\bra \beta', \gamma'\ket} \bra \beta', \gamma'\ket  f ^{\beta'}_{\bf s}(Z)  f ^{\gamma'}_{\bf s}(Z) d\log Z(\beta')  
\end{align*}
where in the last term we sum over decompositions with $\beta', \gamma'$ not equal to $\alpha_j - \alpha_k, \alpha_k - \alpha_i$ for $k \neq i, j$. Note that we have
\begin{align*}
&\Big( d\log Z(\alpha_j - \alpha_k) - d\log Z(\alpha_k - \alpha_i)\Big) \wedge d\log Z(\alpha_j - \alpha_i)\\
&= d\log Z(\alpha_j - \alpha_k) \wedge d\log Z(\alpha_k - \alpha_i).
\end{align*}
It follows that 
\begin{align*}
&(-1)^{\bra \alpha_j, \alpha_i\ket} \bra \alpha_j, \alpha_i\ket d f ^{\alpha_j - \alpha_i}_{\bf s}(Z) \wedge d\log Z(\alpha_j - \alpha_i)\\
&= - \sum_{k \neq i, j} (-1)^{\bra \alpha_j, \alpha_i\ket + \bra \alpha_j - \alpha_k, \alpha_k - \alpha_i\ket} \bra \alpha_j, \alpha_i\ket\bra \alpha_j - \alpha_k, \alpha_k - \alpha_i\ket \\ 
& f ^{\alpha_j - \alpha_k}_{\bf s}(Z) f ^{\alpha_k - \alpha_i}_{\bf s}(Z) d\log Z(\alpha_j - \alpha_k) \wedge d\log Z(\alpha_k - \alpha_i)\\
&- \sum_{\alpha_j - \alpha_i = \beta' + \gamma'} (-1)^{\bra \alpha_j, \alpha_i\ket+ \bra \beta', \gamma'\ket} \bra \alpha_j, \alpha_i\ket \bra \beta', \gamma'\ket  f ^{\beta'}_{\bf s}(Z)  f ^{\gamma'}_{\bf s}(Z)\\
& d\log Z(\beta')\wedge d\log Z(\alpha_j - \alpha_i). 
\end{align*}
where in the last term we sum over decompositions with $\beta', \gamma'$ not equal to $\alpha_j - \alpha_k, \alpha_k - \alpha_i$ for $k = 1, \ldots, n$.
Thus if we have
\begin{align*}
&(-1)^{\bra \alpha_j, \alpha_i\ket + \bra \alpha_j - \alpha_k, \alpha_k - \alpha_i\ket} \bra \alpha_j, \alpha_i\ket\bra \alpha_j - \alpha_k, \alpha_k - \alpha_i\ket \\ &= (-1)^{\bra \alpha_j, \alpha_k\ket + \bra \alpha_k, \alpha_i\ket} \bra \alpha_j, \alpha_k\ket \bra \alpha_k, \alpha_i\ket
\end{align*}
for $k \neq i, j$ and $\bra \beta', \gamma'\ket  f ^{\beta'}_{\bf s}(Z)  f ^{\gamma'}_{\bf s}(Z)$ lies in $({\bf s})^p$ then the $x_j$ component of $F(\nabla^{r, \zeta}_{\bf s})(x_{\alpha_i})$ vanishes modulo terms which are of order at least $p$ in $\bf s$. 
\end{proof}
\begin{rmk} We can choose the exponent $p = p_{ij}$ as a function of $i, j$ (so we get different vanishing conditions of the various curvature components). Note also that the quadratic condition \eqref{quadraticCondition} involves only our choice of basis $\alpha_i$ and the Euler form, while the vanishing condition \eqref{vanishingCondition} involves the invariants $\dt_{\A}$ through the holomorphic generating functions.
\end{rmk}
The quadratic equations appearing in Lemma \ref{nablarzetaFlat} can be rephrased as follows. 

\begin{lem} Fix $p \geq 3$. Let $\alpha_i$ be a basis of $K(\A)\otimes \R$. We denote by $\epsilon_{ij}$ a skew-symmetric tensor with $\epsilon_{i j} = \pm 1$, and by $\lambda$ a fixed arbitrary constant. Then we have the following equivalence: for all pairwise distinct $i, j, k$
\begin{equation*}
\len(\alpha_j - \alpha_i) < p \Rightarrow \bra \alpha_j, \alpha_i\ket\bra \alpha_j - \alpha_k, \alpha_k - \alpha_i\ket = \bra \alpha_j, \alpha_k\ket \bra \alpha_k, \alpha_i\ket
\end{equation*}
if and only if  
\begin{equation*}
\len(\alpha_j - \alpha_i) < p \Rightarrow \begin{cases}
\bra \alpha_i, \alpha_j \ket = \epsilon_{i j} \lambda \quad \text{and}\\
1 + \epsilon_{i j} \epsilon_{j k} + \epsilon_{j i} \epsilon_{i k} + \epsilon_{i k}\epsilon_{k j} = 0. \end{cases}
\end{equation*}
A particular solution is given by choosing $\epsilon_{i j} = -1$ for all $i < j$ such that $\len(\alpha_j - \alpha_i) < p$.
\end{lem}
\begin{proof} We set $x_{i j} = \bra \alpha_i, \alpha_j\ket$, so $x_{ij} = - x_{ji}$. The quadratic equations hold if and only if 
\begin{equation}\label{quadratic}
x^2_{ij} + x_{ij} x_{jk} + x_{ji} x_{ik} + x_{ik} x_{kj} = 0
\end{equation}
for all pairwise distinct $i, j, k$. Cyclically permuting $i \to j \to k$ in \eqref{quadratic} and subtracting from \eqref{quadratic} gives 
\begin{equation*}
x^2_{ij} - x^2_{jk} = 0
\end{equation*}
for all pairwise distinct $i, j, k$, so we must have $x_{ij} = \epsilon_{ij} \lambda$ for a skew-symmetric tensor $\epsilon_{ij}$ and a fixed, arbitrary constant $\lambda$. Plugging this into \eqref{quadratic} turns it into
\begin{equation}\label{quadraticEps}
1 + \epsilon_{i j} \epsilon_{j k} + \epsilon_{j i} \epsilon_{i k} + \epsilon_{i k}\epsilon_{k j} = 0.
\end{equation}
Direct computation shows that a skew-symmetric index $\epsilon_{ij}$ with $\epsilon_{ij} = -1$ for all $i < j$ is a solution.
\end{proof}
\begin{rmk} We may regard \eqref{quadraticEps} abstractly as a system of quadratic constraints on a skew-symmetric tensor $\epsilon_{ij} = \pm 1$, without reference to a basis $\alpha_i$ for $K(\A)\otimes \R$. Many other solutions are possible, e.g. when $\operatorname{rank}(K(\A)) = 3$ the possibilities are
\begin{equation*}
\epsilon_{ij} = \left(\begin{matrix} 
   & & \\
1 & & \\
1 & 1 &
\end{matrix}\right),
\left(\begin{matrix} 
   & & \\
-1 & & \\
1 & 1 &
\end{matrix}\right),
\left(\begin{matrix} 
   & & \\
1 & & \\
1 & -1 &
\end{matrix}\right)   
\end{equation*} 
up to overall multiplication by $\pm 1$. We will consider these solutions further in Section \ref{examplesSec}. Note that when $\operatorname{rank}(K(\A)) = 2$ the condition \eqref{quadraticEps} is empty.
\end{rmk}

Similarly we take a closer look at the vanishing condition \eqref{vanishingCondition}. At least in the simplest case $p = 3$ (i.e. when we are only looking at quadratic jets) there is a natural, simpler condition which implies it, and which does not involve the Euler form or $\dt_{\A}$ invariants.
\begin{lem}\label{vanishingLem} Let $\alpha_i$ be a basis of $K(\A)\otimes \R$. Suppose that for all $i, j = 1, \ldots, n$, $i \neq j$ we have either
\begin{enumerate}
\item[$\bullet$] $\alpha_j - \alpha_i$ is the class of a simple object or its shift, or
\item[$\bullet$] $\alpha_j - \alpha_i$ is the sum of two classes of simple objects or their shifts of the form $\alpha_j - \alpha_k$, $\alpha_k - \alpha_i$, or
\item[$\bullet$] $\alpha_j - \alpha_i$ is not the sum of two classes of simple objects or their shifts.
\end{enumerate}
Then the vanishing condition in Lemma \ref{nablarzetaFlat} holds for $p = 3$ and all $i, j = 1, \ldots, n$.
\end{lem}
\begin{proof} This is obvious from the definition of the grading by length.
\end{proof}
If $[S_i]$ is the basis of $K(\A)$ given by classes of simple objects, another basis $\alpha_j$ for $K(\A)\otimes \R$ satisfies the assumptions of Lemma \ref{vanishingLem} if and only if for all $i \neq j$ we have either 
\begin{equation*}
\alpha_i - \alpha_j = \pm [S]
\end{equation*} 
for some simple $S$, or 
\begin{equation*}
\alpha_i - \alpha_j = (\pm [S]) + (\pm [T])
\end{equation*} 
where, for some $k$,
\begin{equation*}
\pm [S] = \alpha_i - \alpha_k, \quad \pm [T] = \alpha_k - \alpha_j.
\end{equation*}
This condition clearly depends only on the rank of $K(\A)$.\\ 
\begin{exm} Let $\operatorname{rank}(K(\A)) = n$. If $[S_i]$ are the classes of simple objects, a possible solution $\alpha_j$ to the conditions of Lemma \ref{vanishingLem} is given by the \emph{triangular basis}
\begin{equation*}
\alpha_j = \sum^{n}_{i = j}{[S_i]}, 
\end{equation*}
so e.g. in ranks $2, 3$ we have
\begin{equation*}
\begin{matrix}
\alpha_1 = [S_1] + [S_2],   &  &\alpha_1 = [S_1] + [S_2] + [S_3],\\
\alpha_2 = [S_2]                &   &\alpha_2 = [S_2] + [S_3],\\
                                           &  &\alpha_3 = [S_3].    
\end{matrix}
\end{equation*} 
Another possible solution is given by 
\begin{equation*}
\alpha_i = (-1)^{i-1}[S_i] + [S_n],\, i = 1, \ldots, n-1; \,\, \alpha_n = [S_n]
\end{equation*}
so e.g. in rank $3$ we have
\begin{align*}
\alpha_1 &= [S_1] + [S_3],\\
\alpha_2 &= - [S_2] + [S_3],\\
\alpha_3 &= [S_3].    
\end{align*}
\end{exm}

\begin{lem}\label{VFlat} Suppose that the conditions of Lemma \ref{nablarzetaFlat} hold for fixed $i, j = 1, \ldots, n$ and $p \geq 3$. Then we have $g(x_{\alpha_j}, \nabla^{r, \zeta}_{\bf s}(\V^{\zeta}_{\bf s}) x_{\alpha_i}) = 0$ modulo terms which are of order at least $p$ in $\bf s$.   
\end{lem}
\begin{proof} We compute
\begin{align*}
\pi^{\zeta}(\V_{\bf s}(x_{\alpha_l})) &= \pi^{\zeta} \Big(\sum_{\alpha \neq 0}  f ^{\alpha}_{\bf s}(Z)(-1)^{\bra \alpha, \alpha_l\ket} \bra \alpha, \alpha_l\ket x_{\alpha + \alpha_l} \Big)\\
&= \sum^n_{k = 1}  f ^{\alpha_k - \alpha_l}_{\bf s}(Z)(-1)^{\bra \alpha_k, \alpha_l\ket} \bra \alpha_k, \alpha_l\ket x_{\alpha_k}.
\end{align*} 
So in the local trivialisation of $K(\zeta)$ given by $x_{\alpha_1}, \ldots, x_{\alpha_n}$ the endomorphism $\V^{\zeta}_{\bf s}$ is given by the skew-symmetric matrix
\begin{equation}\label{Vzeta}
(\V^{\zeta}_{{\bf s}})_{kl} = (-1)^{\bra \alpha_k, \alpha_l\ket} \bra \alpha_k, \alpha_l\ket  f ^{\alpha_k - \alpha_l}_{\bf s}(Z).
\end{equation}
We have
\begin{align*}
\nabla^{r, \zeta}_{\bf s}(\V^{\zeta}_{\bf s}) &= d\V^{\zeta}_{\bf s} + [A, \V^{\zeta}_{\bf s}]\\
&= d(\V^{\zeta}_{\bf s})_{k l} + \sum^n_{p = 1} (A_{k p}(\V^{\zeta}_{\bf s})_{pl} - (\V^{\zeta}_{\bf s})_{k p} A_{pl})\\
&= (-1)^{\bra \alpha_k, \alpha_l\ket} \bra \alpha_k, \alpha_l \ket d f ^{\alpha_k - \alpha_l}_{\bf s}(Z)\\ 
&+ \sum^n_{p = 1} (-1)^{\bra \alpha_k, \alpha_p\ket + \bra \alpha_p, \alpha_l\ket} \bra \alpha_k, \alpha_p\ket \bra \alpha_p, \alpha_l\ket  f ^{\alpha_k - \alpha_p}_{\bf s}(Z) f ^{\alpha_p - \alpha_l}(Z)\\
&\Big(d\log Z(\alpha_k - \alpha_p) - d\log Z(\alpha_p - \alpha_l)\Big) 
\end{align*}
using the explicit form of $A_{ij}$, $(\V^{\zeta}_{{\bf s}})_{ij}$ found in Lemma \ref{nablarzetaFlat} and above. Arguing as in the proof of Lemma \ref{nablarzetaFlat} we see that if the conditions of that Lemma are satisfied for fixed $i, j = 1, \ldots, n$ and $p$ then $\nabla^{r, \zeta}_{\bf s}(\V^{\zeta}_{\bf s})$ vanishes modulo terms which are of order at least $p$ in ${\bf s}$.
\end{proof}
The following result is straightforward.
\begin{lem}\label{additionalFlatness} Suppose that $\zeta$ is of the form \eqref{specialZeta} (so $\zeta$ is a section of $K(\zeta)$ and the latter is preserved by $C$ and $\U$). Then we have
\begin{align*} 
\nonumber \nabla^{r, \zeta}_{\bf s} (C|_{K(\zeta)}) &=0,\\
\nonumber [C|_{K(\zeta)}, \U|_{K(\zeta)}] &= 0,\\
\nabla^{r, \zeta}_{\bf s}(\U|_{K(\zeta)}) - [C|_{K(\zeta)}, \V^{\zeta}_{\bf s}] + C|_{K(\zeta)} &= 0.
\end{align*}
Moreover $g|_{K(\zeta)}$ is covariantly constant with respect to $\nabla^{r, \zeta}_{\bf s}$, and $C|_{K(\zeta)}$, $\U|_{K(\zeta)}$ are symmetric and $\V^{\zeta}_{\bf s}$ is skew-symmetric with respect to $g$.
\end{lem}
Lemmas \ref{nablarzetaFlat}, \ref{VFlat} and \ref{additionalFlatness} immediately imply a result about Frobenius type structures.
\begin{cor}\label{approxFrobType} Pick a section $\zeta$ of the form \eqref{specialZeta} and suppose that the conditions of Lemma \ref{nablarzetaFlat} hold for all $i, j = 1, \ldots, n$, with the same $p \geq 3$. Then  
\begin{enumerate}
\item[$\bullet$] $- C_{\bullet}(\zeta)\!: T U \to K(\zeta)$ is an isomorphism, 
\item[$\bullet$] the structure on $K(\zeta)$ given by
\begin{equation*}
(\nabla^{r, \zeta}_{\bf s}, C|_{K(\zeta)}, \U|_{K(\zeta)}, \V^{\zeta}_{\bf s}, g|_{K(\zeta)})
\end{equation*}
is a Frobenius type structure modulo terms which are of order at least $p$ in ${\bf s}$. More precisely the conditions $F(\nabla^{r, \zeta}_{\bf s}) = 0$ and $\nabla^{r, \zeta}_{\bf s}(\V^{\zeta}_{\bf s})= 0$ hold as identities of formal power series in ${\bf s}$, modulo terms in ${\bf s}$ which lie in $({\bf s})^p$, while the remaining conditions \eqref{FrobTypeCond1} and those on the metric $g|_{K(\zeta)}$ hold automatically to all orders in ${\bf s}$. 
\end{enumerate} 
\end{cor}
\begin{definition} We call a section $\zeta$ for which the conclusions of Corollary \ref{approxFrobType} apply a \emph{good section}, and we say that the heart $\A$ is \emph{good} if there exists a good section $\zeta\!: U \subset \stab(\A) \to K$. Similarly we say that the basis of $K(\A)\otimes \R$ underlying $\zeta$ is a \emph{good basis}. 
\end{definition}  
Thus Corollary \ref{approxFrobType} gives a characterization of good bases and sections. The finite-dimensional Frobenius type structures obtained via Corollary \ref{approxFrobType} are only approximate, that is they are order $p$ jets of families of Frobenius type structures on the bundle $K(\zeta)$. Before we may apply Hertling's result (Theorem \ref{HertThm}) we need to consider the problem of lifting them to genuine Frobenius type structures. This problem will be solved in the next section. 

\begin{rmk}\label{zetaDependence}
The structure on $K(\zeta)$ specified by $(\nabla^{r, \zeta}_{\bf s},C|_{K(\zeta)}, \U|_{K(\zeta)}, \V^{\zeta}_{\bf s},$ $ g|_{K(\zeta)})$ depends on the choice of a section $\zeta$ of the form \eqref{specialZeta} such that the conditions of Lemma \ref{nablarzetaFlat} hold for all $i, j = 1, \ldots, n$. The section $\zeta$ encodes moduli given by the choice of basis $\alpha_i$ for $K(\A) \otimes \R$ (satisfying the strong constraints of Lemma \ref{nablarzetaFlat}), as well as those given by the choice of holomorphic functions $c_i(Z, {\bf s})$. However it is clear from the results of this section that the structure only depends on the choice of basis. This is in contrast to the Frobenius manifolds we will construct by using Theorem \ref{HertThm}, which will depend on the $c_i(Z, {\bf s})$ moduli as well (through the pullback along $- C_{\bullet}(\zeta)\!: T U \to K(\zeta)$).
\end{rmk}
\section{Lifting to finite-dimensional Frobenius type structures}\label{liftSection}
This section is devoted to proving the following result.
\begin{prop}\label{liftingProp} The approximate Frobenius type structure given by Corollary \ref{approxFrobType} can be lifted canonically to a genuine Frobenius type structure. In other words the solutions to the equations $F(\nabla^{r, \zeta}_{\bf s}) = 0$ and \eqref{FrobTypeCond1} given by Corollary \ref{approxFrobType}, which are defined modulo $({\bf s})^p$, can be lifted canonically to solutions to all orders in ${\bf s}$, and these lifted formal power series solutions converge provided $|| {\bf s} ||$ is sufficiently small. Moreover the conditions on the metric $g|_{K(\zeta)}$ are also preserved.
\end{prop}
The proof of Proposition \ref{liftingProp} uses the general theory of isomonodromy for a family of meromorphic connections on $\PP^1$ with poles divisor $2\cdot 0 + 1\cdot \infty$, and in particular the relevant notions of Stokes factors, matrices and multipliers (see e.g. \cite[{Section\,2}]{bt_stokes}).

We consider the family of meromorphic connections on the holomorphically trivial vector bundle on $\PP^1$ modelled on $K(\zeta)$ (more precisely, a fibre of the trivial bundle $K(\zeta)$) given by
\begin{equation*} 
\nabla^{\zeta}_{\bf s}(Z) = d + \Big( \frac{\U(Z)}{z^2} - \frac{\V^{\zeta}_{\mathbf{s}}(Z)}{z} \Big) d z
\end{equation*}
with parameter space $U \subset \stab(\A)$. This induces a family of connections on the holomorphically trivial principal bundle on $\PP^1$ with fibre $GL(K(\zeta)\pow{\bf s})$.
\begin{definition} Let $P$ be the holomorphically trivial principal bundle on $\PP^1$ with fibre the complex affine algebraic group $GL(K(\zeta)\pow{\bf s} / ({\bf s})^p)$ corresponding to the $GL(K(\zeta)\pow{\bf s})$-bundle described above. 

We define the family of connections $\nabla^{\zeta}_{{\bf s}, p}(Z)$ on $P$ as the reduction modulo $({\bf s})^p$ of the connections $\nabla^{\zeta}_{\bf s}(Z)$, that is
\begin{equation*}
\nabla^{\zeta}_{{\bf s}, p}(Z) = d + \Big( \frac{\U(Z)}{z^2} - \frac{\V^{\zeta}_{\mathbf{s}, p}(Z)}{z} \Big) d z
\end{equation*}
where $\V^{\zeta}_{\mathbf{s}, p} \in \mathfrak{gl}(K(\zeta)\pow{\bf s} / ({\bf s})^p)$ is the reduction modulo $({\bf s})^p$ of $\V^{\zeta}_{\mathbf{s}}$.
\end{definition}
\begin{lem} The family of connections on $P$ given by $\nabla^{\zeta}_{{\bf s}, p}(Z)$ has constant generalised monodromy as $Z$ varies in $U$.
\end{lem}
\begin{proof} By Corollary \ref{approxFrobType} we have $F(\nabla^{r, \zeta}_{{\bf s}, p}) = 0$ and the equations \eqref{FrobTypeCond1} hold in the bundle $P$. This can be stated equivalently by introducing a connection on the pullback of $P$ to $U \times \PP^1$, given by 
\begin{equation*}
\nabla^{r, \zeta}_{{\bf s}, p} - \frac{1}{z} dZ + \Big( \frac{\U(Z)}{z^2} - \frac{\V^{\zeta}_{\mathbf{s}, p}(Z)}{z} \Big) d z,
\end{equation*}
which is then flat. It is well known that flatness of this connection is precisely the isomonodromy condition (see e.g. \cite[Section 3.3]{bt_stokes}).
\end{proof}
Let us focus for a moment on the special case $p = 3$, i.e. quadratic jets of Frobenius type structures. For generic $Z$ the generalised monodromy of $\nabla^{\zeta}_{{\bf s}, 3}(Z)$ can be computed explicitly. We introduce the set of roots (eigenvalues of $\ad( \U )$)
\begin{equation*}
\{ Z(\alpha_i  -  \alpha_j), i \neq j\} \subset \C 
\end{equation*}
and assume its elements are distinct. We write $E_{ij}$ for the elementary matrices. Finally we introduce the function
\begin{align*}
M_2(z_1, z_2) &= 2\pi i \int_{[0, z_1 + z_2]} \frac{dt}{t - z_1}.
\end{align*} 
\begin{lem}\label{StokesLem} The generalised monodromy of the connection $\nabla^{\zeta}_{{\bf s}, 3}(Z)$ is given by 
\begin{enumerate}
\item[$\bullet$] the Stokes rays  
\begin{equation*}
\ell_{ij}(Z) = \R_{> 0}  Z(\alpha_i - \alpha_j)  \subset \C^*
\end{equation*}
for $i \neq j$,
\item[$\bullet$] the corresponding Stokes factors 
\begin{align}\label{approxStokesFactors}
\nonumber \calS_{ij}(Z) = \calS_{\ell_{ij}}(Z) = I &- 2\pi i \, (\V^{\zeta}_{{\bf s}, 3})_{ij} E_{ij}\\
 &+ \sum_k M_2(Z(\alpha_i - \alpha_k), Z(\alpha_k - \alpha_j)) (\V^{\zeta}_{{\bf s}, 3})_{i k } (\V^{\zeta}_{{\bf s}, 3})_{k j} E_{ij}.
\end{align}
\end{enumerate}
 \end{lem}
\begin{proof} The result follows from the Theorem in \cite[Section 4.5]{bt_stokes} applied to the bundle $P$. 
\end{proof}
The following result allows to compute the monodromy modulo $({\bf s})^{3}$ in the situation of Lemma \ref{vanishingLem}.
\begin{cor}\label{simpleStokesCor} Suppose $\alpha_i - \alpha_j$ is the class of a simple object or the sum of classes of simple objects of the form $\alpha_i - \alpha_k$, $\alpha_k - \alpha_j$. Then we have
\begin{equation*}
\calS_{ij}(Z) = I - (-1)^{\bra \alpha_i, \alpha_j\ket} \bra \alpha_i, \alpha_j\ket \dt_{\A}(\alpha_i - \alpha_j, Z) {\bf s}^{\alpha_i - \alpha_j} E_{ij}.
\end{equation*}
\end{cor}
\begin{proof} If $\alpha_i - \alpha_j$ is the class of a simple object or its shift then according to \eqref{approxStokesFactors} we have modulo $({\bf s})^3$
\begin{align*}
\calS_{ij}(Z) &= I - 2\pi i (-1)^{\bra \alpha_i, \alpha_j \ket} \bra \alpha_i, \alpha_j\ket f^{\alpha_i - \alpha_j}_{\bf s}(Z) E_{ij}\\
&= I - (-1)^{\bra \alpha_i, \alpha_j\ket} \bra \alpha_i, \alpha_j\ket \dt_{\A}(\alpha_i - \alpha_j, Z) {\bf s}^{\alpha_i - \alpha_j} E_{ij}.
\end{align*}
In the other case we have similarly modulo $({\bf s})^3$
\begin{align*}
\calS_{ij}(Z) = I & - 2\pi i (-1)^{\bra \alpha_i, \alpha_j \ket} \bra \alpha_i, \alpha_j\ket f^{\alpha_i - \alpha_j}_{\bf s}(Z) E_{ij}\\
&+  (-1)^{\bra \alpha_i , \alpha_k \ket + \bra \alpha_k, \alpha_j\ket} \bra \alpha_i , \alpha_k \ket \bra \alpha_k, \alpha_j\ket \\
&M_2(Z(\alpha_k - \alpha_i), Z(\alpha_j - \alpha_k))f^{\alpha_i - \alpha_k}_{\bf s}(Z) f^{\alpha_k - \alpha_j}_{\bf s}(Z) E_{ij}.
\end{align*}
Let $\log(z)$ denote the branch of the complex logarithm branched along $[0, +\infty)$. According to the formulae for holomorphic generating functions in \cite{bt_stab} we have modulo $({\bf s})^3$
\begin{align*}
f^{\alpha_i - \alpha_j}_{\bf s}(Z) = &\frac{1}{2\pi i} \dt(\alpha_i - \alpha_j, Z) {\bf s}^{\alpha_i - \alpha_j}\\
& + \frac{1}{(2\pi i)^2} (-1)^{\bra \alpha_i - \alpha_k, \alpha_k - \alpha_j \ket}\bra \alpha_i - \alpha_k, \alpha_k - \alpha_j \ket\\ &\quad M_2(Z(\alpha_i - \alpha_k), Z(\alpha_k - \alpha_j))\\& \quad \dt_{\A}(\alpha_i - \alpha_k, Z) \dt_{\A}(\alpha_k - \alpha_j, Z) {\bf s}^{\alpha_i - \alpha_j}.
\end{align*}
On the other hand we have modulo $({\bf s})^3$
\begin{equation*}
f^{\alpha_i - \alpha_k}_{\bf s}(Z) f^{\alpha_k - \alpha_j}_{\bf s}(Z) = \frac{1}{(2\pi i)^2} \dt_{\A}(\alpha_i - \alpha_k, Z) \dt_{\A}(\alpha_k - \alpha_j, Z) {\bf s}^{\alpha_i - \alpha_j}.
\end{equation*}
Moreover the quadratic condition gives 
\begin{align*}
& (-1)^{\bra \alpha_i, \alpha_j \ket} \bra \alpha_i, \alpha_j\ket (-1)^{\bra \alpha_i - \alpha_k, \alpha_k - \alpha_j \ket}\bra \alpha_i - \alpha_k, \alpha_k - \alpha_j \ket \\
&=  (-1)^{\bra \alpha_i , \alpha_k \ket + \bra \alpha_k, \alpha_j\ket} \bra \alpha_i , \alpha_k \ket \bra \alpha_k, \alpha_j\ket.  
\end{align*}
The claim follows.
\end{proof}
\begin{exm}\label{A2} Suppose $\A = \A(A_2, 0)$ for the quiver $A_2 = \bullet \to \bullet$. The simple objects are $S_1 = \C \to 0$, $S_2 = 0 \to \C$. Since $\operatorname{rank}(K(\A)) = 2$ the quadratic conditions \eqref{quadraticCondition} are empty. Setting $p = 3$, a basis $\alpha_i$ for $K(\A)$ satisfying the vanishing conditions \eqref{vanishingCondition} is found by applying Lemma \ref{vanishingLem}. As already observed we may choose $\alpha_1 = [S_1] + [S_2]$, $\alpha_2 = [S_2]$. According to Corollary \ref{simpleStokesCor} we have
\begin{align*}
\calS_{12} &= I - (-1)^{\bra \alpha_1, \alpha_2\ket} \bra \alpha_1, \alpha_2\ket \dt_{\A}(\alpha_1 - \alpha_2, Z) {\bf s}^{\alpha_1 - \alpha_2} E_{12}\\
&= I - (-1)^{\bra [S_1], [S_2] \ket} \bra [S_1], [S_2] \ket s_1 E_{12}\\
&= I - s_1E_{12}. 
\end{align*}
This example can be readily adapted to the generalised Kronecker quiver $K_{\lambda}$ with $\lambda > 1$ arrows $\xymatrix{ \bullet \ar^{\lambda}[r] & \bullet}$, giving
\begin{equation*}
\calS_{12} = I + (-1)^{\lambda} \lambda s_1 E_{12}.
\end{equation*}
\end{exm}
\begin{exm}\label{A3} Suppose $\A = \A(A_3, 0)$ where $A_3 = \bullet \to \bullet \to \bullet$. The simple objects are $S_1 = \C \to 0 \to 0$, $S_2 = 0 \to \C \to 0$, $S_3 = 0 \to 0 \to \C$. Let $E = \C \to \C \to 0$ be the unique extension between $S_1$ and $S_2$. In particular $S_2$ is a subrepresentation of $E$. We choose the basis $\alpha_1 = [S_1] + [S_2] + [S_3]$, $\alpha_2 = [S_2] + [S_3]$, $\alpha_3 = [S_3]$ so by Lemma \ref{vanishingLem} the vanishing conditions \eqref{vanishingCondition} for $p = 3$ are satisfied. The quadratic conditions \eqref{quadraticCondition} also hold since $\bra \alpha_i, \alpha_j \ket = -1$ for all $i < j$. According to Corollary \ref{simpleStokesCor} we have
\begin{align*}
\calS_{12} &= I - (-1)^{\bra [S_1], [S_2]\ket} \bra [S_1], [S_2]\ket s_1 E_{12},\\
&= I - s_1 E_{12},\\ 
\calS_{23} &= I - (-1)^{\bra [S_2], [S_3]\ket} \bra [S_2], [S_3]\ket s_2 E_{23}\\
&= I - s_2 E_{23}
\end{align*}
and 
\begin{align*}
\calS_{13}(Z) &= I - (-1)^{\bra [S_2], [S_3] \ket} \bra [S_2], [S_3] \ket \dt_{\A}([S_1] + [S_2], Z) s_1 s_2 E_{13}\\
& = I - \dt_{\A}([S_1] + [S_2], Z) s_1 s_2 E_{13}.  
\end{align*}
More generally for the quiver $\xymatrix{\bullet\ar^{\lambda}[r] & \bullet\ar^{\lambda}[r] & \bullet}$ we find
\begin{align*}
\calS_{12} &= I + (-1)^{\lambda} \lambda s_1 E_{12},\\
\calS_{23} &= I + (-1)^{\lambda} \lambda s_2 E_{23}\\
\calS_{13}(Z) &= I + (-1)^{\lambda} \lambda \dt_{\A}([S_1] + [S_2], Z) s_1 s_2 E_{13}.
\end{align*}
\end{exm}

There is an analogue of Lemma \ref{StokesLem} which holds for any $p\geq 3$. This follows at once from the explicit formulae proved in \cite[Section 4.5]{bt_stokes}.
\begin{lem}\label{StokesLem_p}
The generalised monodromy of $\nabla^{\zeta}_{{\bf s}, p}(Z)$ is given by the Stokes rays $\ell_{ij}(Z) = \R_{> 0}  Z(\alpha_i - \alpha_j)  \subset \C^*$, for $i \neq j$, together with the corresponding Stokes factors $\calS_{\ell_{ij}}(Z)$ which are the reduction modulo $({\bf s})^p$ of
	\begin{align}
	I -& 2\pi i(\V^{\zeta}_{{\bf s},p})_{ij} E_{ij} +\notag\\
	-& \sum_{m\geq 1}\sum_{k_1\neq\dots\neq k_m} M_{m+1}(Z(\alpha_i - \alpha_{k_1}), \dots, Z(\alpha_{k_m} - \alpha_j)) \notag\\
	&(\V^{\zeta}_{{\bf s},p})_{i k_1 } \cdots (\V^{\zeta}_{{\bf s},p})_{k_m j} E_{ij},
	\end{align}
where $M_{m+1}: (\C^*)^{m+1} \to \C$ are the iterated integrals defined in \cite[Definition 4.4]{bt_stokes}.
\end{lem}

In the rest of this section we let $\calS_{\ell}(Z)$ denote the matrices of Lemma \ref{StokesLem_p} corresponding to a choice of central charge $Z$. 
\begin{cor} Let $V \subset \C^*$ be a convex open sector. 
\begin{enumerate}
\item[$\bullet$] The clockwise ordered product
\begin{equation*}
\prod^{\to}_{\ell \subset V} \calS_{\ell}(Z) \in GL(K(\zeta)\pow{\bf s}/{(\bf s)}^p)
\end{equation*}
is constant as a function of $Z$ as long as the rays $\ell(Z)$ do not cross $\del V$. 
\item[$\bullet$] The Stokes multiplier of the connection $\nabla^{\zeta}_{{\bf s}, p}(Z)$ with respect to the admissible ray $\R_{> 0}$, which by definition is given by the clockwise ordered product
\begin{equation}\label{StokesMult}
\calS = \prod^{\to}_{\ell \subset \bar{\calH}} \calS_{\ell}(Z) \in GL(K(\zeta)\pow{\bf s}/{(\bf s)}^p),
\end{equation}
is in fact constant as a function of $Z \in U \subset \stab(\A)$.
\end{enumerate}
\end{cor}
\begin{proof} Both statements are well-known characterisations of isomonodromy, see e.g. \cite[Sections 2.7, 2.8]{bt_stokes}.
\end{proof}
\begin{definition}\label{StokesMatrix} The \emph{canonical lift} $\tilde{\calS}^0$ of the Stokes multiplier $\calS$ given by \eqref{StokesMult} to $GL(K(\zeta)[{\bf s}])$ is $\calS(Z)$ regarded as an element of $GL(K(\zeta)[{\bf s}])$. Note that we have $\tilde{\calS}^0 |_{{\bf s} = 0} = I$.  
\end{definition}
\begin{exm}\label{A2Stokes} In Example \ref{A2} we find 
\begin{equation*}
\tilde{\calS}^0 = \calS_{12} = \left(\begin{matrix} 1 & (-1)^{\lambda} \lambda s_1 \\ & 1\end{matrix}\right).
\end{equation*}
Evaluating at the special point ${\bf s}_J$ gives the Cartan matrix of the underlying quiver. The $A_2$ case corresponds to $\lambda = 1$ and the matrix
\begin{equation*}
\left(\begin{matrix} 1 & -1 \\ & 1\end{matrix}\right).
\end{equation*}
\end{exm}
\begin{exm}\label{A3Stokes} Similarly in Example \ref{A3} we can compute $\tilde{\calS}^0$ assuming $E$ is unstable, so
\begin{align*}
\tilde{\calS}^0 &= \calS_{23} \calS_{13} \calS_{12}\\
&= \Big(I + (-1)^{\lambda} \lambda s_2 E_{23}\Big)\\
&\quad\,\, \Big(I + (-1)^{\lambda} \lambda  s_1 E_{12}\Big)\\
&=\left(\begin{matrix} 
1&  (-1)^{\lambda}\lambda s_1& 0\\
  & 1 & (-1)^{\lambda} \lambda s_2 \\
  & & 1
\end{matrix}\right).
\end{align*}
In order to check that this agrees with the calculation when $E$ is stable we need to recall the well-known fact that in this case
\begin{equation*}
\dt_{\A}\big([S_1] + [S_2], Z\big) = (-1)^{\lambda-1} \chi(\PP^{\lambda - 1}) = - (-1)^{\lambda} \lambda.
\end{equation*}
from which
\begin{align*}
\tilde{\calS}^0 &= \calS_{12} \calS_{13} \calS_{23}\\
&= \Big(I + (-1)^{\lambda} \lambda s_1 E_{12}\Big)\\
&\quad\,\, \Big(I - (-1)^{2\lambda} \lambda^2 s_1 s_2 E_{13})\\
&\quad\,\, \Big(I + (-1)^{\lambda} \lambda s_2 E_{23}\Big)\\
&=\left(\begin{matrix} 
1& (-1)^{\lambda}\lambda s_1& 0\\
  & 1 & (-1)^{\lambda} \lambda s_2 \\
  & & 1
\end{matrix}\right).
\end{align*} 
Evaluating at the special point ${\bf s}_J$ gives the Cartan matrix of the underlying quiver. The $A_3$ case corresponds to $\lambda = 1$ and the matrix
\begin{equation*}
\left(\begin{matrix} 
1& -1& 0\\
  & 1 & -1 \\
  & & 1
\end{matrix}\right).
\end{equation*}
\end{exm}
\begin{exm} We can repeat the calculations in Examples \ref{A2}, \ref{A3} for the same quivers with \emph{opposite} orientations. Alternatively we may think of this as computing for the same orientation with an opposite choice of triangular basis $\alpha_i = \sum^{n+1 - i}_{r = 1} [S_r]$ (i.e. a different section $\zeta$). The upshot in any case is respectively
\begin{equation*}
\tilde{\calS}^0 = \left(\begin{matrix} 1 & -(-1)^{\lambda}\lambda s_1 \\ & 1\end{matrix}\right), \quad \tilde{\calS}^0 = \left(\begin{matrix} 
1& -(-1)^{\lambda}\lambda s_1& -(-1)^{\lambda}\lambda^2 s_1 s_2\\
  & 1 & -(-1)^{\lambda}\lambda s_2 \\
  & & 1
\end{matrix}\right).
\end{equation*}
\end{exm}
There is also a finite set of non-canonical, but natural lifts of the Stokes matrix. To define these we lift each Stokes factor $\calS_{\ell}(Z) \in GL(K(\zeta)\pow{\bf s}/{(\bf s)}^p)$ trivially to $\tilde{\calS}_{\ell}(Z) \in GL(K(\zeta)[{\bf s}])$, and take the product of lifts 
\begin{equation}\label{naturalLift}
\prod^{\to}_{\ell \subset \bar{\calH}} \tilde{\calS}_{\ell}(Z) \in GL(K(\zeta)[{\bf s}]).
\end{equation} 
We regard \eqref{naturalLift} as a function of $Z$. It is constant modulo $({\bf s})^p$, but the higher order terms can take finitely many distinct values on different chambers of $\stab(\A)$. 
\begin{definition}\label{naturalLiftDef}
A \emph{natural lift} $\tilde{\calS}$ of the Stokes multiplier $\calS$ given by \eqref{StokesMult} to $GL(K(\zeta)[{\bf s}])$ is either the canonical lift $\tilde{\calS}^0$, or one of the finitely many possible values of the product \eqref{naturalLift}. Note that all the natural lifts agree modulo $({\bf s})^p$.
\end{definition}
\begin{exm}\label{A4Lift} Let $\A = \A(A_4, 0)$ where $A_4 = \bullet \to \bullet \to \bullet \to \bullet$. We work with $p = 3$. Choose the triangular basis $\alpha_i = \sum^4_{r = i} [S_r]$. This is good since $\bra \alpha_i, \alpha_j \ket = -1$ for all $i < j$. Two possible determinations of the product \eqref{naturalLift} in different chambers are
\begin{align*}
\tilde{\calS}_{34} \tilde{\calS}_{23} \tilde{\calS}_{12} &= (I - s_3 E_{34})(I-s_2 E_{23})(I - s_1E_{12}) = \left(\begin{matrix} 
1 & - s_1 & 0 & 0 \\
   & 1  & -s_2 & 0 \\
   &     & 1& - s_3 \\
   &   &     & 1
\end{matrix}\right)  
\end{align*}
and 
\begin{align*}
&\tilde{\calS}_{12} \tilde{\calS}_{13} \tilde{\calS}_{23} \tilde{\calS}_{24} \tilde{\calS}_{34}\\
&= (I - s_1 E_{12})(I - s_1 s_2 E_{13})(I - s_2 E_{23})( I - s_2 s_3 E_{24})(I - s_3 E_{34})\\
&=\left(\begin{matrix}
1 & - s_1 & 0 & s_1 s_2 s_3 \\
   & 1  & -s_2 & 0 \\
   &     & 1& - s_3 \\
   &   &     & 1
\end{matrix}
\right).
\end{align*}
Indeed one can check that these are the only values of \eqref{naturalLift} on $\stab(\A)$.
\end{exm}
\begin{prop} For fixed $Z$ and sufficiently small $|| {\bf s}||$ there is a canonical choice of a connection $\tilde{\nabla}^{\zeta}_{\bf s}(Z)$ on the trivial principal $GL(K(\zeta))$-bundle, of the form
\begin{equation*} 
\tilde{\nabla}^{\zeta}_{\bf s}(Z) = d + \Big( \frac{\U(Z)}{z^2} - \frac{\tilde{\V}^{\zeta}_{\mathbf{s}}(Z)}{z} \Big) d z
\end{equation*}
with Stokes multiplier with respect to the admissible ray $\R_{> 0}$ given by the canonical lift $\tilde{\calS}^0$. The connection matrix $\tilde{\V}^{\zeta}_{\mathbf{s}}(Z)$ is skew-symmetric and depends holomorphically on both $Z$ and ${\bf s}$. The reduction of $\tilde{\nabla}^{\zeta}_{\bf s}(Z)$ modulo $({\bf s})^p$ is $\nabla^{\zeta}_{{\bf s}, p}(Z)$. The same holds for any other choice of a natural lift $\tilde{\calS}$.
\end{prop}
\begin{proof} The result follows immediately from the Theorem in \cite[Section 4.8]{bt_stokes}.  
\end{proof}

\begin{definition} The canonical lift of the approximate Frobenius type structure $(\nabla^{r, \zeta}_{\bf s}, C|_{K(\zeta)}, \U|_{K(\zeta)}, \V^{\zeta}_{\bf s}, g|_{K(\zeta)})$ is defined as the collection of holomorphic objects
\begin{enumerate}
\item[$\bullet$] the connection $\nabla^r$ given by 
\begin{equation*}
\tilde{\nabla}^{r, \zeta}_{\bf s} = d + \tilde{A}
\end{equation*}
with connection form $\tilde{A}_{i j} = (\tilde{\V}^{\zeta}_{\bf s})_{i j}\,d \log Z(\alpha_i - \alpha_j)$,
\item[$\bullet$] the Higgs field $C$, endomorphism $\U$ and metric $g$ given by the restrictions $C|_{K(\zeta)}, \U|_{K(\zeta)}, g|_{K(\zeta)}$,
\item[$\bullet$] the endomorphism $\V$ given by $\tilde{\V}^{\zeta}_{\bf s}$. 
\end{enumerate}
Of course one can give an identical definition for any other choice of a natural lift $\tilde{\calS}$.
\end{definition}
\begin{cor}\label{polyLiftCor} The collection $(\tilde{\nabla}^{r, \zeta}_{\bf s}, C|_{K(\zeta)}, \U|_{K(\zeta)}, \tilde{\V}^{\zeta}_{\bf s}, g|_{K(\zeta)})$ is a Frobenius type structure on the bundle $K(\zeta) \to U$, depending holomorphically on ${\bf s}$ for $|| {\bf s}||$ sufficiently small. The same holds for any other choice of a natural lift $\tilde{\calS}$.
\end{cor}
\begin{proof} For fixed ${\bf s}$, with $|| {\bf s} ||$ sufficiently small, the family of connections $\tilde{\nabla}^{\zeta}_{\bf s}(Z)$ has constant generalised monodromy as $Z$ varies in $U$. By a well-known characterisation of isomonodromy (see e.g. \cite[Section 3.3]{bt_stokes}), the family of connections on $P$ pulled back to $U \times \PP^1$
\begin{equation*}
\tilde{\nabla}^{r, \zeta}_{{\bf s}} - \frac{1}{z} dZ + \Big( \frac{\U(Z)}{z^2} - \frac{\tilde{\V}^{\zeta}_{\mathbf{s}}(Z)}{z} \Big) d z
\end{equation*}
is flat. This is equivalent to the equations $F(\tilde{\nabla}^{r, \zeta}_{{\bf s}}) = 0$ and \eqref{FrobTypeCond1}. The conditions on $g$ can be checked directly. 
\end{proof}
\begin{proof}[Proof of Proposition \ref{liftingProp}] This follows at once from Corollary \ref{polyLiftCor}.
\end{proof}
\section{Pullback to Frobenius manifolds}\label{pullbackSection}
Let $K(\zeta) \to U$ the bundle constructed in Section \ref{projectionSec} (where $U \subset \stab{\A}$ denotes an open subset where $-dZ(\zeta)$ is an isomorphism, as usual). Under suitable assumptions Corollary \ref{approxFrobType} endows $K(\zeta)$ with a Frobenius type structure defined modulo terms which lie in $({\bf s})^p$. In this section we fix a natural lift of this jet to a genuine Frobenius type structure, depending holomorphically on ${\bf s}$ in a sufficiently small polydisc $\Delta$ (see Proposition \ref{liftingProp} and Corollary \ref{polyLiftCor}). We denote this lifted, genuine Frobenius type structure, depending holomorphically on the parameters ${\bf s}$, by  
\begin{equation}\label{liftedFrobTypeStr} 
(\tilde{\nabla}^{r, \zeta}_{\bf s}, C|_{K(\zeta)}, \U|_{K(\zeta)}, \tilde{\V}^{\zeta}_{\bf s}, g|_{K(\zeta)}).
\end{equation} 

We briefly recall the main ingredients in a Frobenius manifold, and then apply Theorem \ref{HertThm} to endow $U \subset \stab(\A)$ with the structure of a semisimple Frobenius manifold (with Euler field and flat identity).
\begin{definition}\label{FrobDefinition} A Frobenius manifold is a complex manifold $M$ such that the fibres of the holomorphic tangent bundle $\tg{M}$ are endowed with a commutative, associative product $\circ$. Moreover we assume that there are a unit field $e$, a (Euler) field $E$, and a nondegenerate holomorphic quadratic form $g_M$ on the fibres of $\tg{M}$ (the metric) such that the following conditions hold.
\begin{enumerate}
\item[$\bullet$] The metric $g_M$ is flat. Denote its Levi-Civita connection by $\nabla^{g_M}$.
\item[$\bullet$] Introducing a Higgs field $C^M$ on $M$ by $C^M_X(Y) = - X \circ Y$, we have $\nabla^{g_M} (C^M) = 0$.
\item[$\bullet$] The unit field $e$ is flat, i.e. $\nabla^{g_M}(e) = 0$.
\item[$\bullet$] Taking Lie derivatives along the Euler field we have $\LL_E(\circ) = \circ$ and $\LL_E(g_M) = (2 - d) g_M$ for some $d \in \C$ ($2-d$ is called the conformal dimension of $M$).
\item[$\bullet$] We have $g_M(C^M_X Y, Z) = g_M(Y, C^M_X Z)$, that is the metric is compatible with the multiplication $\circ$.
\end{enumerate}
$M$ is called semisimple if $\circ$ is semisimple. Local coordinates which correspond to a semisimple basis for $\circ$ are known as \emph{canonical coordinates}. 
\end{definition} 
A Frobenius manifold $M$ gives rise to a Frobenius type structure on $\tg{M}$.

\begin{lem}[{\cite[Lemma 5.11]{hert}}]\label{HertLem} Let $M$ be a Frobenius manifold (with flat identity $e$ and Euler field $E$ of conformal dimension $2-d$. Define
\begin{enumerate}
\item[$\bullet$] a Higgs field $C^M$ by $C^M_X(Y) = - X \circ Y$ as above,
\item[$\bullet$] endomorphisms $\U^M$ and $\V^M$ by 
\begin{align*}
\U^M &= E \circ ( \bullet ),\\
\V^M &= \nabla^g_{ \bullet } E - \frac{2 - d}{2} I.
\end{align*}
\end{enumerate}
\noindent Then $(\tg{M}, \nabla^{g_M}, C^M, \U^M, \V^M, g_M)$ is a Frobenius type structure on the fibres of $\tg{M}$.
\end{lem}

Suppose $(\nabla^r, C, \U, \V, g)$ is a Frobenius type structure on an auxiliary bundle $K \to M$ with a section $\zeta$. Denote by $v$ the natural morphism $ - C_{\bullet}(\zeta)\!: \tg{M} \to K$. Theorem \ref{HertThm} can then be restated by saying that, under the assumptions spelled out in the Theorem, the pullbacks 
\begin{align*}
&\nabla^{r, M} = v^{-1} \nabla^r v, \quad C^M = v^{-1} C v,\\
&\quad \U^M = v^{-1} \U v, \quad \V^M = v^{-1} \V v,\\
&\quad \quad \quad g_M = g( v( - ), v( - ))
\end{align*} 
define a Frobenius type structure on the tangent bundle $\tg{M}$, and moreover that this Frobenius type structure comes from a genuine Frobenius manifold as in Lemma \ref{HertLem}. According to \cite[Lemma 4.1]{hert}  the multiplication $\circ$ on $\tg{M}$ is given by
\begin{equation*}
X \circ Y = - C^M_X Y = v^{-1} (C_{X} v(Y)),
\end{equation*}
while the flat identity $e$ is uniquely defined by requiring 
\begin{equation*}
C_e = - I. 
\end{equation*}
Note that the multiplication $\circ$ is uniquely characterised by the property
\begin{equation*}
C_X C_Y = - C_{X \circ Y}.  
\end{equation*}
As explained in the statement of \cite[Theorem 5.12]{hert} and its proof, the flat identity is in fact given by
\begin{equation*}
e = v^{-1}(\zeta),
\end{equation*}
while the Euler field is
\begin{equation*}
E = \U(e).
\end{equation*}
It satisfies $\LL_E(g_M) = (2 - d) g_M$, where $d$ is the constant in Theorem \ref{HertThm}.

We now turn to our Frobenius type structure \eqref{liftedFrobTypeStr} (depending holomorphically on ${\bf s} \in \Delta$). Recall that it depends on $\zeta$ only through the choice of basis $\alpha_i$ for $K(\A) \otimes \R$ (see Remark \ref{zetaDependence}). 
\begin{lem} The endomorphism $\tilde{\V}^{\zeta}_{\bf s}$ acts on the space of flat sections of $\tilde{\nabla}^{r, \zeta}_{\bf s}$ on $U \subset \stab(\A)$. The spectrum of $\tilde{\V}^{\zeta}_{\bf s}$ is constant on $U$, i.e. in the $Z$ direction (it depends highly nontrivially on ${\bf s} \in \Delta$). 
\end{lem} 
\begin{proof} These are well-known consequences of isomonodromy, see e.g. \cite[Lecture 3]{dubrovin}.
\end{proof}

Fix the choice of basis $\alpha_i$ for $K(\A) \otimes \R$. Let $\frac{d}{2}$ be an eigenvalue of $\tilde{\V}^{\zeta}_{\bf s}$ acting on the space of flat sections of $\nabla^{r, \zeta}_{\bf s}(\zeta)$ on $U$. Then we can find a section $\zeta$ of $K \to U$ such that 
\begin{equation}\label{zetaExistence}
\begin{cases}
\nabla^{r, \zeta}_{\bf s}(\zeta) = 0,\\
\tilde{\V}^{\zeta}_{\bf s}(\zeta) = \frac{d}{2} \zeta.
\end{cases}
\end{equation}
Restricting $U$ if necessary we may assume that $dZ(\zeta)$ is still an isomorphism. We can now summarise all our results so far.

\begin{thm}\label{mainThm} Let $\frac{d({\bf s})}{2}$ be an eigenvalue of $\tilde{\V}^{\zeta}_{\bf s}$.  There exists a semisimple Frobenius manifold structure on $U \subset \stab(\A)$ such that
\begin{enumerate}
\item[$\bullet$] the canonical coordinates are given by 
\begin{equation*}
u_i = Z(\alpha_i),
\end{equation*}
\item[$\bullet$] the flat identity and Euler field are
\begin{equation*}
e = \sum_i \del_{Z(\alpha_i)}, \quad E = \sum_{i} Z(\alpha_i) \del_{Z(\alpha_i)}
\end{equation*}
\item[$\bullet$] the flat metric is given by
\begin{equation*}
g_{\bf s}(u) = \sum_i c^2_i(Z, {\bf s}) du^2_i 
\end{equation*} 
\item[$\bullet$] the conformal dimension is $2 - d({\bf s})$.
\end{enumerate}
It is given by pulling back the Frobenius type structure \eqref{liftedFrobTypeStr} along $dZ(\zeta)$, where $\zeta$ is a section of $K \to U$ as in Corollary \ref{zetaExistence}. 

Moreover if we have $\alpha_i - \alpha_j \in \pm K(\A)_{> 0}$ for all $i \neq j$ then this can be analytically continued to a Frobenius manifold structure on all $\stab(\A)$, without monodromy. 
\end{thm}

In the following we refer to the structure given by Theorem \ref{mainThm} simply as the semisimple  Frobenius manifold structure on $\stab(\A)$ (at the point ${\bf s}$). Notice that flatness of $e$ comes from flatness of $\zeta$ and that $\frac{\del}{\del Z(\alpha_i)}\circ \frac{\del}{\del Z(\alpha_j)} = \delta_{ij} \frac{\del}{\del Z(\alpha_i)}$.

By construction we can easily understand the Stokes multiplier.  
\begin{lem} The Frobenius type structure \eqref{liftedFrobTypeStr} and the semisimple Frobenius manifold structure on $\stab(\A)$ have the same Stokes multiplier $\tilde{\calS}$ (given by Definition \ref{naturalLiftDef}). 
\end{lem}
\begin{proof} This is clear computing in the basis $\del_{\tilde{u}_i} = \big( c_{i}(Z, {\bf s}) \big)^{-1} \del_{u_i}$.
\end{proof}
From our point of view the main object of interest is the section $\zeta$ of the bundle $K$. It determines the conformal dimension and the metric of the Frobenius manifold structure on $\stab(\A)$.  Computing $\zeta = \zeta(\tilde{\calS})$ as a (multi-valued) function of the Stokes multiplier is an instance of the (hard) inverse problem for semisimple Frobenius manifolds (see e.g. \cite{guzzetti}).
\section{Examples}\label{examplesSec}
In this section we discuss several examples to which we may apply the general theory developed so far. We concentrate on the case $p = 3$ for these examples. By Theorem \ref{mainThm} these determine (possibly several) natural families of semisimple Frobenius manifold structures, sharing the same quadratic jet. An important example of higher order jets is discussed in the next Section.
\subsection{A general construction} Fix positive integers $n, \lambda$. Write $[S_k]$ for the standard basis of the lattice $\Z^n$. Choose $n$ linearly independent elements $\alpha_i$ of $\Z^n$ such that for all $i \neq j$ we have either
\begin{enumerate}
\item[$\bullet$] $\alpha_i - \alpha_j = \pm [S_k]$ for some $k$, or
\item[$\bullet$] $\alpha_i - \alpha_j = (\pm [S_u]) + (\pm [S_v])$ for some $u, v$, and we have $[S_u] = \alpha_i -\alpha_k$, $[S_v] = \alpha_k - \alpha_j$, or
\item[$\bullet$] $\alpha_i - \alpha_j$ is not of the form $\pm [S_k]$ or $(\pm [S_u]) + (\pm [S_v])$. 
\end{enumerate}
Pick a skew-symmetric tensor $\epsilon_{ij}$, $1 \leq i, j \leq n$, with values in $\{\pm 1\}$, giving a solution to the quadratic equations \eqref{quadratic}.

Let $Q$ be a quiver with $n$ vertices and let $\A = \A(Q) \subset \D(Q)$. Then $K(\A)$ is identified with $\Z^n$ and the canonical basis $[S_i]$ is the basis of classes of simple objects. Similarly the lattice elements $\alpha_i$ are canonically identified with elements of $K(\A)$. Since $\alpha_i$ is a basis there exist (possibly several) skew-symmetric bilinear forms $\bra - , - \ket$ on $K(\A)$ such that  
\begin{equation}\label{prescribedEuler}
\len(\alpha_j - \alpha_i) < 3 \Rightarrow \bra \alpha_i, \alpha_j\ket = \epsilon_{i j} \lambda.
\end{equation}
Our general construction implies the following.
\begin{lem}\label{generalExample} Let $Q$ be a quiver with Euler form $\bra - , - \ket$ satisfying \eqref{prescribedEuler}. In particular we can choose $Q$ as the quiver with vertices labelled by $[S_i]$ and with $\bra [S_i], [S_j] \ket$ arrows between $[S_i]$, $[S_j]$ for $i\leq j$.

Then the assumptions of Theorem \ref{mainThm} hold and so there is a canonical family of Frobenius manifold structures on $\stab(\A)$ for each choice of a natural lift $\tilde{\calS}$.
\end{lem}
\subsection{Fixed triangular basis} We run the construction above with $\lambda = 1$ and the \emph{fixed} choice of triangular basis $\alpha_i = \sum^n_{r = i} [S_r]$
(i.e. the $i$th basis element $\alpha_i$ is the $i$th row of the upper triangular rank $n$ matrix given by $\tau_{n, ij} = 1$ for $i \leq j$). The tables (Figures \ref{triangularRank2} - \ref{triangularRank4}) list the quivers for which this is a good basis (i.e. the basis underlying a good section) up to rank $4$ together with the jet of the corresponding Stokes matrix. In the case of rank $4$ we only list half the solutions, up to reversing all arrows (which acts nontrivially on the Stokes matrix). The rank $4$ case contains a free parameter $\kappa = \bra \alpha_1, \alpha_4\ket$ since $\len(\alpha_1 - \alpha_4) = 3$.
\begin{figure}[ht]
{\small
\begin{equation*} 
\begin{xy} 0;<.3pt,0pt>:<0pt,-.3pt>:: 
(0,3) *+{1} ="1",
(133,3) *+{2} ="2",
"1", {\ar"2"},
\end{xy} 
\quad \left(\begin{matrix}
1 & -s_1\\
   &  1   
\end{matrix}\right)
\quad\quad 
\begin{xy} 0;<.3pt,0pt>:<0pt,-.3pt>:: 
(0,3) *+{1} ="1",
(133,3) *+{2} ="2",
"2", {\ar"1"},
\end{xy} 
\quad \left(\begin{matrix}
1 & s_1\\
   &  1    
\end{matrix}\right)
\end{equation*}
}
\caption{Good quivers for $\tau_2$ and their Stokes matrix.} 
\label{triangularRank2}
\end{figure}

\begin{figure}[ht]
{\small
\begin{equation*}
\begin{matrix}
\begin{xy} 0;<.2pt,0pt>:<0pt,-.2pt>:: 
(0,3) *+{1} ="1",
(133,3) *+{2} ="2",
(261,3) *+{3} ="3",
"1", {\ar"2"},
"2", {\ar"3"},
\end{xy}
& \begin{xy} 0;<.2pt,0pt>:<0pt,-.2pt>:: 
(0,3) *+{1} ="1",
(133,3) *+{2} ="2",
(270,3) *+{3} ="3",
"2", {\ar"1"},
"3", {\ar"2"},
\end{xy}
& \begin{xy} 0;<.2pt,0pt>:<0pt,-.2pt>:: 
(0,3) *+{1} ="1",
(136,3) *+{2} ="2",
(271,3) *+{3} ="3",
"2", {\ar"1"},
"2", {\ar"3"},
\end{xy}
\\
\\
\left(\begin{matrix}
1 & -s_1& \\
   &  1   & -s_2\\
   &        & 1
\end{matrix}\right) & \left(\begin{matrix}
1 & s_1& s_1 s_2\\
   &  1   & s_2\\
   &        & 1
\end{matrix}\right) & \left(\begin{matrix}
1 & s_1 & - s_1s_2 \\
   &  1    & - s_2\\
   &        &1
\end{matrix}\right)\\
\\
\begin{xy} 0;<.2pt,0pt>:<0pt,-.2pt>:: 
(0,3) *+{1} ="1",
(131,3) *+{2} ="2",
(271,3) *+{3} ="3",
"1", {\ar"2"},
"3", {\ar"2"},
\end{xy} &
\begin{xy} 0;<.2pt,0pt>:<0pt,-.2pt>:: 
(0,3) *+{1} ="1",
(140,133) *+{2} ="2",
(273,3) *+{3} ="3",
"1", {\ar"2"},
"3", {\ar|*+{\scriptstyle 2}"1"},
"2", {\ar"3"},
\end{xy}
&
\begin{xy} 0;<.2pt,0pt>:<0pt,-.2pt>:: 
(0,3) *+{1} ="1",
(130,127) *+{2} ="2",
(271,3) *+{3} ="3",
"2", {\ar"1"},
"1", {\ar|*+{\scriptstyle 2}"3"},
"3", {\ar"2"},
\end{xy}\\
\left(\begin{matrix}
1 & -s_1 &  \\
   &  1    & s_2\\
   &        &1
\end{matrix}\right)&
\left(\begin{matrix}
1 & s_1 & \\
   &  1 & -s_2 \\
   &     & 1 
\end{matrix}\right)&
\left(\begin{matrix}
1 & -s_1 & -s_1 s_2 \\
   &  1 & s_2 \\
   &     & 1 
\end{matrix}\right)

\end{matrix}
\end{equation*}
}
\caption{Good quivers for $\tau_3$ and their Stokes matrix.}
\label{triangularRank3}
\end{figure}

\begin{figure}[ht]
{\tiny
\begin{equation*}
\begin{matrix}
\xymatrix{
1\ar[dd]|{-1+\kappa}\ar[ddrr]|{1-\kappa}\ar[rr] & & 2\ar[dd]\\
   &  &   \\
4 &  &3\ar[ll]
}&
\xymatrix{
1\ar[dd]|{-1+\kappa}\ar[ddrr]|{1-\kappa} & & 2\ar[dd]\ar[ll]\\
   &  &   \\
4 &  &3\ar[ll]
}&
\xymatrix{
1\ar[dd]|{-1+\kappa}\ar[ddrr]|{-1-\kappa}\ar[rr] & & 2\ar[dd]\\
   &  &   \\
4 &  &3\ar[ll]
}\\
\\
\left(\begin{matrix}
1 &  -s_1 & \\
   &  1 &  - s_2 \\
   &     & 1  & - s_3\\
    &   &     &  1       
\end{matrix}\right) & \left(\begin{matrix}
1 &  s_1 & -s_1 s_2\\
   &  1 &  - s_2 \\
   &     & 1  & - s_3\\
    &   &     &  1       
\end{matrix}\right) & \left(\begin{matrix}
1 &  s_1 & \\
   &  1 &  - s_2 \\
   &     & 1  & - s_3\\
    &   &     &  1       
\end{matrix}\right)\\
\\
\xymatrix{
1\ar[dd]|{-1+\kappa}\ar[ddrr]|{3-\kappa} & & 2\ar[ll]\\
   &  &   \\
4 &  &3\ar[ll]\ar[uu] 
}&
\xymatrix{
1\ar[dd]|{1+\kappa}\ar[ddrr]|{1-\kappa} & & 2\ar[ll]\ar[dd]\\
   &  &   \\
4\ar@{>>}[uurr]|\hole &  &3\ar[ll] 
}&
\xymatrix{
1\ar[dd]|{1+\kappa}\ar[ddrr]|{1-\kappa}\ar[rr] & & 2\\
   &  &   \\
4 &  &3\ar[ll]\ar[uu] 
}\\
\\
\left(\begin{matrix}
1 &  -s_1 & - s_1 s_2 \\
   &  1 & s_2 & -s_2 s_3\\
   &     & 1  & - s_3\\
    &   &     &  1       
\end{matrix}\right) & \left(\begin{matrix}
1 &  -s_1 & s_1 s_2 \\
   &  1 & s_2 & \\
   &     & 1  & - s_3\\
    &   &     &  1       
\end{matrix}\right) & \left(\begin{matrix}
1 &  -s_1 &  \\
   &  1 & s_2 & - s_2 s_3\\
   &     & 1  & - s_3\\
    &   &     &  1       
\end{matrix}\right)\\
\\
\xymatrix{
1\ar[dd]|{1+\kappa}\ar[ddrr]|{-1-\kappa}\ar[rr] & & 2\ar[dd]\\
   &  &   \\
4\ar@{>>}[uurr]|\hole &  &3\ar[ll] 
}&
\xymatrix{
1\ar[dd]|{-1+\kappa}\ar[ddrr]|{1-\kappa} & & 2\ar[ll]\\
   &  &   \\
4 &  &3\ar[ll]\ar[uu] 
}&
\xymatrix{
1\ar[dd]|{1+\kappa}\ar[ddrr]|{-1-\kappa} & & 2\ar[ll]\ar[dd]\\
   &  &   \\
4\ar@{>>}[uurr]|\hole &  &3\ar[ll] 
}\\
\\
\left(\begin{matrix}
1 &  -s_1 &  \\
   &  1 & s_2 &  \\
   &     & 1  & - s_3\\
    &   &     &  1       
\end{matrix}\right) & \left(\begin{matrix}
1 &  s_1 & s_1 s_2 \\
   &  1 & s_2 & -s_2 s_3 \\
   &     & 1  & - s_3\\
    &   &     &  1       
\end{matrix}\right) & \left(\begin{matrix}
1 &  s_1 & s_1 s_2 \\
   &  1 & s_2 &  \\
   &     & 1  & - s_3\\
    &   &     &  1       
\end{matrix}\right)\\
\\
\end{matrix}
\end{equation*}
}
\caption{Good quivers for $\tau_4$ and their Stokes matrix.}
\label{triangularRank4}
\end{figure}
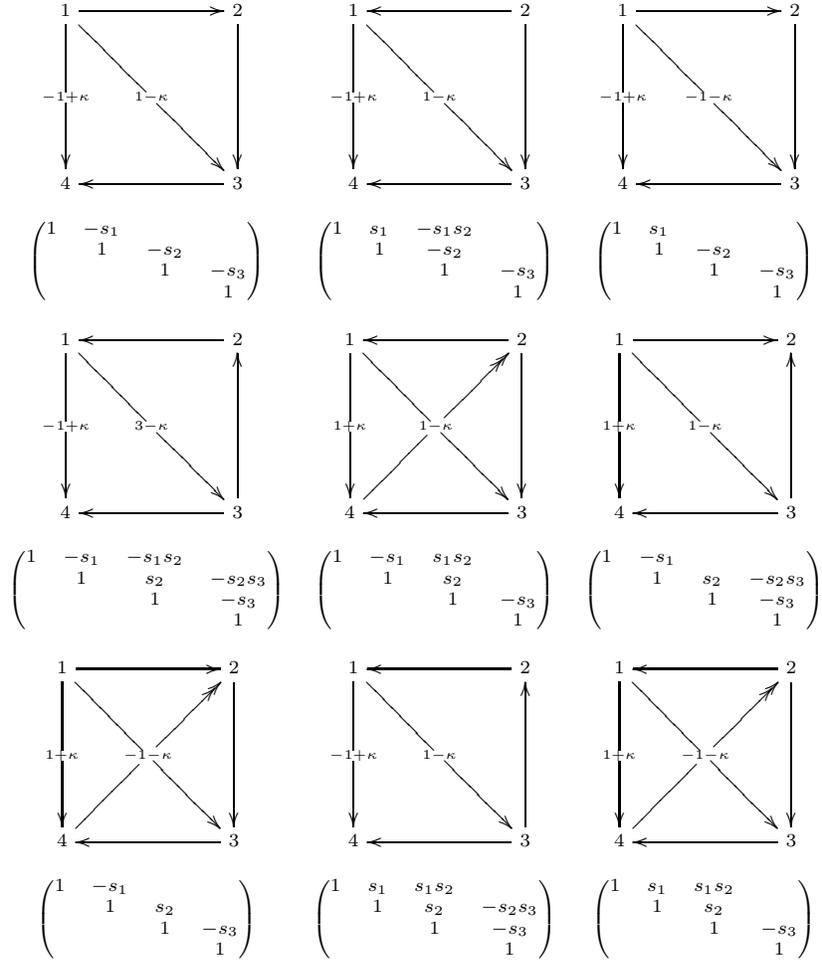
\newpage

\subsection{} A very useful example of a basis of $\Z^3$ satisfying the conditions of Lemma \ref{vanishingLem} is obtained when $\alpha_i$ corresponds to the $i$th row of
\begin{equation*}
\delta_3 = \left(\begin{matrix}
1 & 0  & 1\\
0 & -1 & 1\\
0 & 0  & 1
\end{matrix}\right).
\end{equation*}
Setting $\lambda = 1$ we find $\delta_3$ is a good basis precisely for the quivers in Figure \ref{delta3figure}.
\begin{figure}[ht]
{\small
\begin{equation*}
\begin{matrix}
\xymatrix{
1\ar[dr] & &2\ar[ll]\\
& 3\ar[ur] &
}&
\xymatrix{
1\ar[rr]\ar[dr] & &2\\
  & 3\ar[ur] 
}&
\xymatrix{
1 & &2\ar[ll]\\
  & 3\ar[ul]\ar[ur] &
}
\\
\left(\begin{matrix}
1 & -s_1 s_2 & -s_1\\
0 & 1 &0\\
0 &s_2&1
\end{matrix}
\right)&
\left(\begin{matrix}
1 & 0 & -s_1\\
0 & 1 &0\\
0 &s_2&1
\end{matrix}
\right)&
\left(\begin{matrix}
1 & s_1 s_2 & s_1\\
0 & 1 &0\\
0 &s_2&1
\end{matrix}
\right)\\
\\
\xymatrix{
1 \ar[rr]\ar[dr]&&2\ar[dl]\\
&3&
}&
\xymatrix{
1 & & 2\ar[ll]\ar[dl]\\ 
  & 3 \ar[ul]
}&
\xymatrix{
1 \ar[rr]& & 2\ar[dl]\\ 
  & 3 \ar[ul]
}\\
\left(\begin{matrix}
1 & 0 & -s_1\\
0 & 1 &0\\
0 &-s_2&1
\end{matrix}
\right)& 
\left(\begin{matrix}
1 & s_1 s_2 & s_1\\
0 & 1 &0\\
0 &s_2&1
\end{matrix}
\right)&
\left(\begin{matrix}
1 & 0 & s_1\\
0 & 1 &0\\
0 &-s_2&1
\end{matrix}
\right)
\end{matrix}
\end{equation*}
}
\caption{Good quivers for $\delta_3$ and their Stokes matrix.}
\label{delta3figure}
\end{figure}
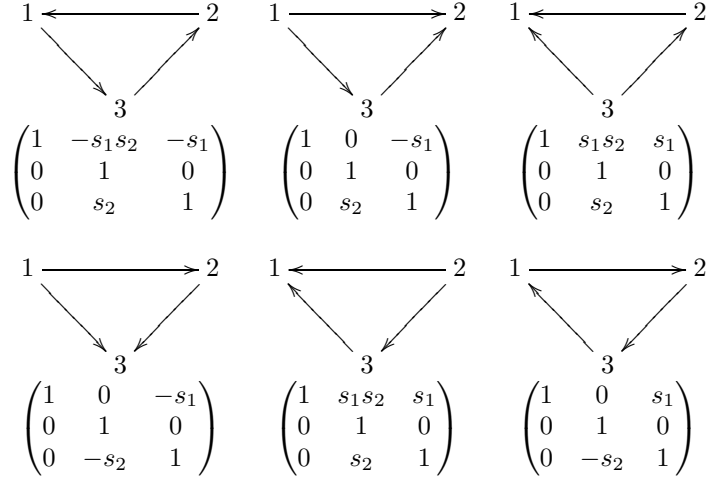

\section{$A_n$ quivers}\label{AnSection}

Examples \ref{A2} and \ref{A3} can be generalised to give a canonical quadratic jet of a family of semisimple Frobenius manifold structures on $\stab(\A(A_n))$. We first show how to do this and then we prove that a slight modification of our construction enhances this quadratic jet to a jet of order $n$.

We run the construction of Section \ref{examplesSec} with $\lambda = 1$, $p = 3$ and the special solutions to the quadratic equations \eqref{quadratic} and vanishing conditions \eqref{vanishingCondition} given by 
\begin{equation*}
\epsilon_{ij} = -1, \, i < j, \quad 
\alpha_i = \sum^n_{r = i} [S_r].
\end{equation*} 
To check that the latter indeed gives a solution we note that 
\begin{align}\label{AnClasses}
\nonumber \alpha_i - \alpha_{i + 1} &= [S_i],\\
\nonumber \alpha_i - \alpha_{i + 2} &= [S_i] + [S_{i + 1}]\\  
\nonumber &= (\alpha_i - \alpha_{i + 1}) + (\alpha_{i +1} - \alpha_{i + 2}),\\
\alpha_i - \alpha_j &= \sum^j_{r = i} [S_i] \neq [S_u] + [S_v], \quad j > i + 2. 
\end{align}
We regard \eqref{prescribedEuler} as a linear system to be solved for $\bra [S_i], [S_j]\ket$ for $i < j$. A little thought shows that in the present case the unique solution is given by  
\begin{equation*}
\bra [S_i], [S_j]\ket = -\delta_{j, i +1}, \quad i = 1, \dots, n - 1.
\end{equation*} 
Therefore the quiver $Q$ of Lemma \ref{generalExample} is $A_n = \bullet \to \cdots \to \bullet$ ($n$ vertices), endowed with the ordered collection of stable objects 
\begin{equation*}
S_i = \cdots \to 0 \to \underbrace{\C}_{i} \to 0 \to \cdots,\, i = 1, \dots, n.
\end{equation*}
According to Theorem \ref{mainThm} this determines (several) families of Frobenius manifolds whose underlying complex manifold is $\stab(\A(A_n))$, one for each choice of a natural lift $\tilde{S}$ (since $\alpha_i - \alpha_j \in K_{>0}(\A)$ for all $i \neq j$, they are well-defined on all $\stab(\A(A_n))$, i.e. the monodromy there is trivial). All these families agree modulo $({\bf s})^3$. We can also compute the Stokes multiplier $\calS$ modulo $({\bf s})^3$. Recall this is constant in $Z$ and so we can compute it assuming that only the simple objects are stable. With this assumption and using \eqref{AnClasses}, Corollary \ref{simpleStokesCor} shows that the only nontrivial Stokes factors $\calS_{ij}$ modulo $({\bf s})^3$ (up to exchanging $i, j$) are
\begin{align*}
\calS_{i, i +1} &= I - (-1)^{\bra [S_i], [S_{i+1}] \ket} \bra [S_i], [S_{i+1}] \ket s_{i} E_{i, i + 1} = I - s_i E_{i, i+1}, 
\end{align*}
so we have
\begin{align}\label{AnStokesMatrix}
 \calS &= \prod^{n - 1}_{j =1} \calS_{n - j, n- j +1} = I - \sum^{n-1}_{i = 1} s_i E_{i, i+1} = \left(\begin{matrix}
1 & -s_1 & \\
   & 1 & -s_2 & \\
   &   &    \ddots     &\\
   &   &         &1& -s_n\\
   &   &          &  & 1
\end{matrix}
\right).
\end{align}

We now show that a slight modification of our construction for $A_n$ resolves the ambiguity in the choice of $\tilde{\calS}$ and gives a single canonical family. The following simple observation is the crucial point.
\begin{lem}\label{AnVanishing} Fix the oriented quiver $A_n$ with triangular basis $\alpha_1, \ldots, \alpha_n$ as above. Let $i, j \in \{1,\ldots, n\}$. Then the quadratic condition \eqref{quadraticCondition} holds and the vanishing condition \eqref{vanishingCondition} holds modulo terms of order $\len(\alpha_i - \alpha_j) + 1$.
\end{lem}
\begin{proof} We have already observed that setting $\epsilon_{h k} = -1$ for $h < k$ gives a solution to \eqref{quadratic} for all $i, j$ (i.e. with $p = n+1$). For the claim concerning the vanishing condition \eqref{vanishingCondition} we need to show 
\begin{equation*} 
\bra \beta, \gamma\ket  f ^{\beta}_{\bf s}(Z)  f ^{\gamma}_{\bf s}(Z) \in ({\bf s})^{\len(\alpha_i - \alpha_j)+1}
\end{equation*}
for all nontrivial decompositions $\alpha_j - \alpha_i = \beta + \gamma$ with $\beta, \gamma$ not equal to $\alpha_j - \alpha_k$, $\alpha_k - \alpha_i$. We can assume $i < j$. For $A_n$ with triangular basis we have $\alpha_i - \alpha_j = \sum^{j - 1}_{k = i} [S_i]$, so a decomposition $\alpha_j - \alpha_i$ corresponds uniquely to a set partition $A \cup B = \{i, i+1, \ldots, j - 1\}$ with $A \cap B = \emptyset$. For such a decomposition we have
\begin{equation*}
 f ^{\beta}_{\bf s}(Z)  f ^{\gamma}_{\bf s}(Z) = f ^{\sum_{h \in A}[S_h]}_{\bf s}(Z)  f ^{\sum_{k \in B}[S_k]}_{\bf s}(Z).
\end{equation*}
By \eqref{fFunFps} we always have $f ^{\sum_{h \in A}[S_h]}_{\bf s}(Z) \in ({\bf s})^{|A|}$. We claim that if $A$ contains a gap (i.e. it is not a subset of consecutive integers in $\{i, i+1, \ldots, j - 1\}$) then in fact $f ^{\sum_{h \in A}[S_h]}_{\bf s}(Z) \in ({\bf s})^{|A| + 1}$. This can be shown by induction on $|A|$, starting from the fact that $f^{[S_h] + [S_k]}_{\bf s}(Z) \in ({\bf s})^3$ if $h, k$ are not consecutive integers. To prove this note that for $h \neq k$ we have
\begin{align*}
f^{[S_h] + [S_k]}_{\bf s}(Z) = &\frac{1}{2\pi i} \dt([S_h] + [S_k], Z) s_h s_k\\
& + \frac{1}{(2\pi i)^2} (-1)^{\bra [S_h], [S_k] \ket}\bra [S_h], [S_k] \ket\\ &\quad M_2(Z([S_h]), Z([S_k]))\\& \quad \dt_{\A}([S_h], Z) \dt_{\A}([S_k], Z) s_h s_k.
\end{align*} 
and for $A_n$ we have $\dt([S_h] + [S_k], Z) = \bra [S_h], [S_k] \ket = 0$ if $h, k$ are not consecutive. 

The argument above applies equally to $B$, so $f ^{\sum_{h \in A}[S_h]}_{\bf s}(Z) f ^{\sum_{k \in B}[S_k]}_{\bf s}(Z) \in ({\bf s})^{|A| + |B|+1}$ unless $A$, $B$ are both sequences of consecutive integers in $\{i, i+1, \ldots, j-1\}$. But this means that up to exchanging $A$, $B$ we have $A = \{i, i+1, \ldots, q\}$, $B = \{q +1, q +2, \ldots, j - 1\}$ for some $q$, that is $\beta = \sum_{h \in A} [S_h] = \sum^{q}_{h = i} [S_i] = \alpha_i - \alpha_{q + 1}$ and similarly $\gamma = \sum_{k \in B} [S_k] = \alpha_{q + 1} - \alpha_j$, a contradiction. 
\end{proof}
Recall the definition of $A_{ij}$ and $\V_{\bf s}^{\zeta}$ in \eqref{Aij} and \eqref{Vzeta} respectively.
\begin{cor}\label{AnConnection} In the case of $A_n$ define a new connection form $A'$ and endomorphism $\V'^{\zeta}_{\bf s}$ by
\begin{equation*}
A'_{ij} = A_{ij} \mod ({\bf s})^{\len(\alpha_i - \alpha_j) + 1},\quad
\V'^{\zeta}_{\bf s} = \V^{\zeta}_{\bf s} \mod ({\bf s})^{\len(\alpha_i - \alpha_j) + 1}.
\end{equation*}
Then the structure on $K(\zeta)$ given by
\begin{equation*}
(d + A', C|_{K(\zeta)}, \U|_{K(\zeta)}, \V'^{\zeta}_{\bf s}, g|_{K(\zeta)})
\end{equation*}
is a Frobenius type structure modulo terms which are of order at least $n + 1$ in ${\bf s}$. In particular its Stokes matrix $\calS$ given by \eqref{AnStokesMatrix} is in fact constant modulo $({\bf s})^{n+1}$, that is the canonical lift $\tilde{\calS}^0$ and all the natural lifts $\tilde{\calS}$ coincide with $\calS$ thought of as an element of $GL(K(\zeta))[{\bf s}]$. 
\end{cor}
\begin{proof} It is enough to prove that the product 
\begin{equation*}
\tilde{\calS} = \prod^{\to}_{\ell \subset \bar{\calH}} \calS_{\ell}(Z)
\end{equation*}
is constant in $Z$ modulo $({\bf s})^{n+1}$. By Lemma \ref{AnVanishing} the $(i,j)$ entry $(\tilde{\calS})_{ij}$ of $\tilde{\calS}$ is constant modulo $({\bf s})^{\len(\alpha_i - \alpha_j) + 1}$. Now choose $0 < q < n - \len(\alpha_i - \alpha_j)$ and look at the component of the polynomial $(\tilde{\calS})_{ij}$ of total degree $\len(\alpha_i - \alpha_j) + q$. By our choice of $A'$, $\V'$ and the explicit formula for Stokes factors in terms of connection coefficients given in \cite[Theorem 4.5]{bt_stokes}  (i.e. the higher order analogue of \eqref{approxStokesFactors}) a contribution to $(\tilde{\calS})_{ij}$ of this degree involves at least $\len(\alpha_i - \alpha_j) + q$ distinct factors in the product and so corresponds to a decomposition of $\alpha_i - \alpha_j$ with at least $\len(\alpha_i - \alpha_j) + q$ nonvanishing summands in $K(\A)$, 
\begin{equation*}
\alpha_i - \alpha_j = \sum^{\len(\alpha_i - \alpha_j) + q}_{h = 1} \gamma_h, \quad \gamma_h \in K(\A) \setminus \{0\}.
\end{equation*}       
It follows that we must have $\gamma_h \in - K_{> 0}(\A)$ for some $h$, a contradiction since we are taking the product over the positive half-plane $\bar{\calH}$. 
\end{proof}
\begin{exm} Let us revisit the case of $A_4$ discussed in Example \ref{A4Lift}. Recall that looking only at the Stokes factors modulo $({\bf s})^3$ we found two different chambers for the product $\tilde{\calS}(Z)$, namely
\begin{equation*}
\tilde{\calS}_{34} \tilde{\calS}_{23} \tilde{\calS}_{12} = {\tiny \left(\begin{matrix} 
1 & - s_1 & 0 & 0 \\
   & 1  & -s_2 & 0 \\
   &     & 1& - s_3 \\
   &   &     & 1
\end{matrix}\right)},\,\,\,
\tilde{\calS}_{12} \tilde{\calS}_{13} \tilde{\calS}_{23} \tilde{\calS}_{24} \tilde{\calS}_{34}={\tiny \left(\begin{matrix}
1 & - s_1 & 0 & s_1 s_2 s_3 \\
   & 1  & -s_2 & 0 \\
   &     & 1& - s_3 \\
   &   &     & 1
\end{matrix}\right)}.
\end{equation*}
The problem is resolved by looking at Stokes factors modulo $({\bf s})^5$. Indeed this gives an additional factor $\tilde{S}_{14}$, and a lengthy direct computation shows $\tilde{S}_{14} = (I - \dt([S_1] + [S_2] + [S_3], Z)) s_1 s_2 s_3$ which contributes to the second factorisation giving
\begin{equation*}
\tilde{\calS}_{12} \tilde{\calS}_{13} \tilde{\calS}_{14}\tilde{\calS}_{23} \tilde{\calS}_{23} \tilde{\calS}_{34}={\tiny \left(\begin{matrix}
1 & - s_1 & 0 & 0 \\
   & 1  & -s_2 & 0 \\
   &     & 1& - s_3 \\
   &   &     & 1
\end{matrix}\right)}
\end{equation*}
as required.
\end{exm}
We can now prove our main result in the case of the standard $A_n$ quiver.
\begin{proof}[Proof of Theorem \ref{mainThmAn}] By Corollary \ref{AnConnection} we have a jet of a family of Frobenius manifold structures on $K(\zeta)$ over $\stab(\A(A_n))$ modulo $({\bf s})^{n + 1}$. This can be lifted canonically since all the natural lifts $\tilde{\calS}$ coincide. 

According to \cite[Corollary 4.7]{dubrovin}  the Stokes matrix $\tilde{\calS}$ given by \eqref{AnStokesMatrix} evaluated at the special point ${\bf s}_J = (s_1 = 1, \cdots, s_n = 1)$ (for a unique choice of eigenvalue $d({\bf s}_J)$) is precisely the Stokes matrix of a branch of the unfolding space of the $A_n$ singularity, so we recover this branch from $\A(A_n)$. 
\end{proof}
The construction of the present Section applies equally to all quivers with the same underlying unoriented graph as $A_n$. The proofs are just the same.
\begin{lem}\label{AnSolutions} Setting $\epsilon_{ij} = \pm 1$ for $i = 1, \cdots, n-1$ and all $i < j$ and evaluating at ${\bf s}_J$ gives $2^{n-1}$ Frobenius manifold structures on $\bar{\calH}^n$, corresponding to the $2^{n-1}$ orientations of the $A_n$ unoriented graph. Their Stokes matrices for $n \leq 4$ are given in Figures \ref{triangularRank2}, \ref{triangularRank3} and \ref{triangularRank4} (choosing $\kappa = \pm 1$ for the latter). 
\end{lem}
Note that all these quivers are mutations of $A_n$ with its standard orientation. It turns out that the corresponding Frobenius manifolds are always related by analytic continuation, i.e. they are different branches of the same semisimple Frobenius structure on $\conf_n(\C)$. We describe this (including more general mutations) in the next Section.

\section{Mutations and analytic continuation}\label{mutationsSec}

In this Section we extend Theorem \ref{mainThmAn} to all mutations of $A_n$, and then provide examples where mutation-equivalence for quivers can be related directly to analytic continuation for semisimple Frobenius manifolds. The main result concerns $A_n$ for $n \leq 5$, although we expect this holds for all $A_n$. We refer to \cite[Section 7]{bridsmith} for basic material on quiver mutation and its categorification, and to \cite[Appendix F]{dubrovin}  and \cite[Section 1.8]{guzzetti} for analytic continuation of semisimple Frobenius manifolds.

If the quivers $(Q_i, W_i)$, $i = 1, 2$ are mutation-equivalent the $\cy$ triangulated categories $\D(Q_i, W_i)$ are equivalent. So if the infinite-dimensional Frobenius type structures $K \to \stab(\A(Q_i, W_i))$ admit good sections, it seems reasonable to expect that such sections $\zeta_{\A(Q_i, W_i)}$ can be chosen so that the corresponding semisimple Frobenius manifolds are equivalent in some sense. Via the canonical coordinates $u_i$ given by Theorem \ref{mainThm}, both structures can be thought of as living naturally on open subsets (isomorphic to $\calH^n$) of the configurations space  
$\conf_n(\C) = \{(u_1, \ldots, u_n) \in \C^n : i \neq j \Rightarrow u_i \neq u_j\} / \Sigma_n$.
Thus a natural equivalence relation is that the two structures should be branches of the same semisimple Frobenius manifold on $\conf_n(\C)$. This can be checked via the Stokes matrices as follows.

Fixing $A \in M_n(\C)$ an upper triangular matrix with eigenvalues $1$, we introduce an elementary braiding matrix $\beta_{i, i+1}(A)$, $i = 1, \ldots, n-1$ by perturbing the identity $I$ with a block
\begin{equation*}
\left(\begin{matrix}
0 & 1\\
1 & -A_{i, i+1}
\end{matrix}
\right)
\end{equation*}   
with upper left entry corresponding to the $i$th diagonal entry, so e.g. 
\begin{equation*}
\beta_{1,2}\left(\begin{matrix}
1 & a\\
0 & 1
\end{matrix}\right) = \left(\begin{matrix}
0 & 1\\
1 & -a
\end{matrix}\right).
\end{equation*}

For the inverse operator $\beta_{i,i+1}^{-1}$ the corresponding block is
	\begin{equation*}
	\begin{pmatrix}
	-A_{i,i+1} & 1 \\ 1 & 0
	\end{pmatrix}.
	\end{equation*}
Setting 
\begin{equation*}
\beta_{i, i+1} . A = \beta_{i, i+1}(A)\,A\,\beta_{i, i+1}(A)
\end{equation*}
defines an action of the braid group $\operatorname{Br}_{n+1} \cong \pi_1(\conf_n(\C))$. 

Fix a semisimple Frobenius structure on an open subset of $\conf_n(\C)$, and let $\calS$ denote the corresponding Stokes matrix. We can assume without loss of generality that $\calS$ is upper triangular. The set of Stokes matrices corresponding to analytic continuations of the structure is precisely the orbit of $\calS$ under the action of $\operatorname{Br}_{n+1}$, combined with the standard action of permutation matrices $P$ and change of sign matrices $I_i$, 
\begin{equation*}
A \mapsto P A P^{-1}, \quad A \mapsto I_i A I^{-1}_i.
\end{equation*}
This is the equivalence relation we check in our examples. 

In the following we fix a reference quiver $Q$ and consider its orbit under mutations. It is important to recall that if $Q_i$, $i = 1, 2$ are mutation equivalent there is a canonical bijection between their vertices and so a canonical bijection between the simple objects of $\A(Q_i)$. We fix once and for all a labelling of the vertices of $Q$, corresponding to a labelling $S_i$ for the simple objects of $A(Q)$. We will also use this specific induced labelling when writing bases of $K(\A(Q_i))$ for mutation-equivalent $Q_i$.  
\subsection{$A_n$ quivers} Let $\mu A_n$ be a quiver in the (finite) mutation orbit of $A_n$. Our aim is to write down good bases of $K(\A(\mu A_n))$ for which the conclusions of Lemma \ref{AnVanishing} and Corollary \ref{AnConnection} hold.  
\begin{enumerate}
\item If the unoriented graph underlying $\mu A_n$ is the same as $A_n$, then choose the basis 
	\begin{equation}\label{alphalinear}
	\alpha_i= \sum_{r=i}^n [S_r],\ i=1,\dots, n.
	\end{equation}
\item If a clockwise oriented triangle appears, 
	
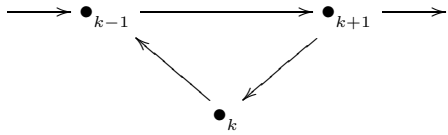
\begin{equation*}
	\xymatrix{
	&\ar[r]&\bullet_{_{k-1}}  \ar[rr]&& \bullet_{_{k+1}} \ar[dl]\ar[r] &\\
	&&&\bullet_{_k} \ar[ul] & &
}
	\end{equation*}
then consider the basis 
	\begin{equation}\label{alphatr}
	\begin{cases}
	\alpha_{i}=\sum_{j=i}^{k-2}[S_{j}] + \alpha_{k-1} \quad\text{ for }i\leq k-2\\
	\alpha_{k-1}= [S_{k-1}]+[S_{k+1}]+\alpha_{k+2}\\
	\alpha_k= -[S_k]+[S_{k+1}]+\alpha_{k+2}\\
	\alpha_{k+1}=[S_{k+1}] +\alpha_{k+2}\\
	\alpha_{i}=\sum_{j=i}^{n}[S_{j}] \quad\text{ for } i\geq k+2.
	\end{cases}
	\end{equation}
\item Triangles can be combined. This is straightforward if they are not overlapping, otherwise take 
\begin{equation*}
	\xymatrix@=13pt{
	 \bullet_{_{k-1}}  \ar[rr]&& \bullet_{_{k+1}}\ar[dl] \ar[rr] && \bullet_{_{k+3}}\ar[dl] \\
	&\bullet_{k} \ar[ul] & & \bullet_{_{k+2}}\ar[ul]
	}\quad
	\begin{cases}
	\alpha_{k-1}=[S_{k-1}]+[S_{k+1}]+[S_{k+3}] \\
	\alpha_k= -[S_k]+[S_{k+1}]+[S_{k+3}] \\
	\alpha_{k+1}= [S_{k+1}]+[S_{k+3}] \\
	\alpha_{k+2} = -[S_{k+2}]+[S_{k+3}]\\
	\alpha_{k+3} = [S_{k+3}].
	\end{cases}
	\end{equation*}
\item The last possible configuration we need to consider is 
	\begin{equation}\label{tr_3code}
	\xymatrix{
	\cdots\bullet_{_{l_i}}\ar[r] &\bullet_{_{k-1}}\ar[rr] && \bullet_{_{k+1}} \ar[dl]\ar[r] &\bullet_{_{r_i}}\cdots\\
	&& \bullet_{_{k}}\ar[ul]&&\\
	&&\bullet_{_{d_i}}\ar[u]\ar@{.}[d]&&\\
	&&&&&
	}
	\end{equation}
An admissible basis is 
	\begin{equation}\label{alphatr_3code}
	\begin{cases}
	\alpha_{r_i} &= \sum_{j\geq i}[S_{r_j}]\\	
	\alpha_{k+1} &= [S_{k+1}] + \alpha_{r_i}\\
	\alpha_{k-1} &= [S_{k-1}]+[S_{k+1}]+\alpha_{r_i}\\
	\alpha_k &= -[S_k]+[S_{k+1}]+\alpha_{r_i}\\
	\alpha_{l_i} &= \sum_{j\geq i}[S_{l_j}]+\alpha_{k-1}\\
	\alpha_{d_i} &= \sum_{j\geq i}[S_{d_j}]+\alpha_k.
	\end{cases}
	\end{equation}
\end{enumerate}  
Arguing precisely as in the proofs of Lemma \ref{AnVanishing} and Corollary \ref{AnConnection} we obtain the following.
\begin{lem}\label{AnConnectionMutation} Let $\mu A_n$ be mutation-equivalent to $A_n$. Then combining the configurations \eqref{alphalinear} - \eqref{alphatr_3code} above yields bases for $K(\A(\mu A_n))$ for which the conclusions of Lemma \ref{AnVanishing} and Corollary \ref{AnConnection} hold. 
\end{lem}
\begin{proof}[Proof of Theorem \ref{mainThmMutations} (i)] This follows immediately from Lemma \ref{AnConnectionMutation}.
\end{proof}
Our next aim is to understand enough of the Stokes matrices of these semisimple Frobenius manifolds in order to prove part $(ii)$ of Theorem \ref{mainThmMutations} (which is restricted at present to the case $n \leq 5$). We present the computations below. In all cases it is possible to choose a central charge $Z$ so that the only stable objects are either simples or extensions between two simples, so Corollary \ref{simpleStokesCor} is sufficient to perform the computation.

\subsubsection*{Mutation classes of $A_2$}
\begin{align*}
A_2&=\xymatrix{\bullet_{_1}\ar[r]&\bullet_{_2}},\ \calS(A_2) =\begin{pmatrix} 1&-s\\0&1\end{pmatrix}\\
\mu_1A_2 &=\xymatrix{\bullet_{_1}&\bullet_{_2}\ar[l]},\ \calS(\mu_1A_2) =\begin{pmatrix}1&s\\0&1\end{pmatrix}.
\end{align*}

\subsubsection*{Mutation classes of $A_3$}
\begin{align*}
A_3 &= \xymatrix{ \bullet_{_1}\ar[r]&\bullet_{_2}\ar[r]&\bullet_{_3}}
 & \calS(A_3) &=\begin{pmatrix}1&-s_1&0\\0&1&-s_2\\0&0&1\end{pmatrix} \\
\mu_3A_3 &= \xymatrix{\bullet_{_1}\ar[r]&\bullet_{_2}&\bullet_{_3}\ar[l]}
& \calS(\mu_3A_3) &=\begin{pmatrix} 1&-s_1&0\\0&1&s_2\\0&0&1\end{pmatrix}\\
\mu_1\mu_3A_3 &= \xymatrix{\bullet_{_1}&\bullet_{_2}\ar[l]&\bullet_{_3}\ar[l]} 
& \calS(\mu_1\mu_3A_3) &=\begin{pmatrix} 1&s_1&s_1s_2\\0&1&s_2\\0&0&1\end{pmatrix}\\
\mu_1A_3 &= \xymatrix{\bullet_{_1}&\bullet_{_2}\ar[l]\ar[r]\ar@{}[dl]& \bullet_{_3}\\ && }
& \calS(\mu_1A_3) &= \begin{pmatrix} 1&s_1&-s_1s_2\\0&1&-s_2\\0&0&1 \end{pmatrix}\\
\mu_2A_3 &= \xymatrix@=15pt{\bullet_{_1}\ar[rr]&&\bullet_{_3}\ar[dl]\\ & \bullet_{_2}\ar[ul]& } 
 & \calS(\mu_2A_3) &=\begin{pmatrix}1&s_1s_2&-s_1\\0&1&0\\0&-s_2&1\end{pmatrix} \\
\mu_2\mu_1\mu_3A_3 &= \xymatrix@=15pt{\bullet_{_1}\ar[dr]&&\bullet_{_3}\ar[ll]\\ & \bullet_{_2}\ar[ur]& }
 & \calS(\mu_2\mu_1\mu_3A_3) &=\begin{pmatrix}1&-s_1s_2&s_1\\0&1&0\\0&-s_2&1\end{pmatrix}
\end{align*}

\subsubsection*{Mutation classes of $A_4$} It is enough to consider\\
$\mu_1A_4=\xymatrix{
	\bullet_{_1}&\bullet_{_2}\ar[l]\ar[r]&\bullet_{_3}\ar[r]&\bullet_{_4} 
	}$ 
	\begin{equation}\label{mu1_A4}
\calS(\mu_1 A_4)= \begin{pmatrix}
	1& s_1 & -s_1s_2&0\\
	0 & 1 & -s_2&0 \\
	0&0&1&-s_3\\
	0&0&0&1
	\end{pmatrix}
	\end{equation}
$\mu_4A_4=\xymatrix{
	\bullet_{_1}\ar[r]&\bullet_{_2}\ar[r]&\bullet_{_3}&\bullet_{_4}\ar[l]
	}$
	\begin{equation}\label{mu4_A4}
\calS(\mu_4 A_4)= \begin{pmatrix}
	1& -s_1 & 0&0\\
	0 & 1 & -s_2&0 \\
	0&0&1&s_3\\
	0&0&0&1
	\end{pmatrix}
	\end{equation}
$\mu_4\mu_2\mu_1A_4=\xymatrix{
	\bullet_{_1}\ar[r]&\bullet_{_2}&\bullet_{_3}\ar[l]&\bullet_{_4}\ar[l]
	}$
\begin{equation}\label{mu421_A4}
	\calS(\mu_4\mu_2\mu_1A_4)= \begin{pmatrix}
	1 & -s_1 & 0 & 0 \\ 0 & 1 & s_2 & s_2s_3 \\ 0 & 0 & 1 & s_3\\ 0 & 0 & 0 & 1
	\end{pmatrix}
	\end{equation}
$\mu_1\mu_4A_4\xymatrix{
	\bullet_{_1}&\bullet_{_2}\ar[l]\ar[r]&\bullet_{_3}&\ar[l]\bullet_{_4}
	}$
	\begin{equation}\label{mu14_A4}
\calS(\mu_1\mu_4 A_4)=\begin{pmatrix}
	1 & s_1 & -s_1s_2 & 0 \\
	0 & 1 & -s_2 & 0 \\
	0 & 0 & 1 & s_3 \\
	0 & 0 & 0 & 1
	\end{pmatrix}
	\end{equation}
$\mu_2\mu_1A_4=\xymatrix{
	\bullet_{_1}\ar[r]&\bullet_{_2}&\bullet_{_3}\ar[l]\ar[r]&\bullet_{_4}
	}$
	\begin{equation}\label{mu21_A4}
\calS(\mu_2\mu_1A_4)= \begin{pmatrix}
	1 & -s_1 & 0 & 0 \\  0 & 1 & s_2 & -s_2s_3 \\ 0 & 0 & 1 & -s_3\\ 0 & 0 & 0 & 1
	\end{pmatrix}
	\end{equation}
$\mu_1\mu_2\mu_1 A_4=\xymatrix{
	\bullet_{_1}&\bullet_{_2}\ar[l]&\bullet_{_3}\ar[l]\ar[r]&\bullet_{_4}
	}$
	\begin{equation}\label{mu121_A4}
	\calS(\mu_1\mu_2\mu_1 A_4) = \begin{pmatrix}
	1 & s_1 & s_1s_2 &-s_1s_2s_3\\
	0 & 1 & s_2 & -s_2s_3\\
	0 & 0 & 1 &-s_3\\
	0 & 0 & 0 &1
	\end{pmatrix}
	\end{equation}
$\mu_4\mu_1\mu_2\mu_1 A_4=\xymatrix{
	\bullet_{_1}&\bullet_{_2}\ar[l]&\bullet_{_3}\ar[l]&\bullet_{_4}\ar[l]
	}$
	\begin{equation}\label{mu4121_A4}
\calS(\mu_4\mu_1\mu_2\mu_1 A_4)= \begin{pmatrix}
	1 & s_1 & s_1s_2 &s_1s_2s_3\\
	0 & 1 & s_2 & s_2s_3\\
	0 & 0 & 1 &s_3\\
	0 & 0 & 0 &1
	\end{pmatrix}
	\end{equation}
$\mu_2A_4=\xymatrix@=15pt{
	\bullet_{_1}\ar[rr]&&\bullet_{_3}\ar[r]\ar[dl]&\bullet_{_4}\\
	& \bullet_{_2}\ar[ul]&&
	}$
	\begin{equation}\label{mu2_A4}
\calS(\mu_2 A_4)= \begin{pmatrix}
	1 & -s_1s_2 & -s_1 &0\\ 0 &1& 0 & 0\\ 0 &s_2 & 1&-s_3\\
	0& 0 & 0 & 1
	\end{pmatrix}
	\end{equation}
$\mu_4\mu_2A_4=
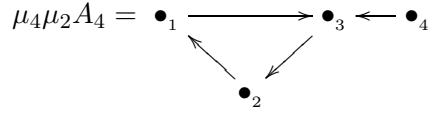
\xymatrix@=15pt{
	\bullet_{_1}\ar[rr]&&\bullet_{_3}\ar[dl]&\bullet_{_4}\ar[l]\\
	& \bullet_{_2}\ar[ul]&&
	}$
	\begin{equation}\label{mu42_A4}
\calS(\mu_4\mu_2 A_4)= \begin{pmatrix}
	1 & -s_1s_2 & -s_1 &0\\ 0 &1& 0 & 0\\ 0 &s_2 & 1& s_3\\
	0& 0 & 0 & 1
	\end{pmatrix}
	\end{equation}
$\mu_3A_4=
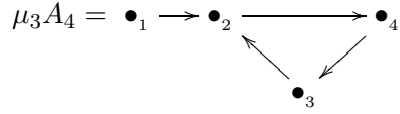
\xymatrix@=15pt{
	\bullet_{_1}\ar[r]&\bullet_{_2}\ar[rr]&&\bullet_{_4}\ar[dl]\\
	&& \bullet_{_3}\ar[ul]&
	}$
	\begin{equation}\label{mu3_A4}
\calS(\mu_3 A_4)= \begin{pmatrix}
	1 & -s_1 & 0 &0\\ 0 &1& -s_2s_3 & -s_2\\ 0 &0 & 1&0\\
	0& 0 & s_3 & 1
	\end{pmatrix}
	\end{equation}
$\mu_1\mu_3A_4=
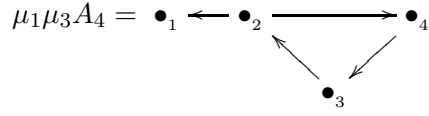
\xymatrix@=15pt{
	\bullet_{_1}&\bullet_{_2}\ar[l]\ar[rr]&&\bullet_{_4}\ar[dl]\\
	&& \bullet_{_3}\ar[ul]&
	}$
	\begin{equation}\label{mu13_A4}
\calS(\mu_1\mu_3 A_4)= \begin{pmatrix}
	1 & s_1 & -s_1s_2s_3 &-s_1s_2\\ 0 &1& -s_2s_3 & -s_2\\ 0 &0 & 1&0\\
	0& 0 & s_3 & 1
	\end{pmatrix}
	\end{equation}
	
\subsubsection*{Mutation classes of $A_5$}
It is enough to consider\\
$\mu_1A_5=\xymatrix{
	\bullet_{_1}&\bullet_{_2}\ar[l]\ar[r]&\bullet_{_3}\ar[r]&\bullet_{_4} \ar[r]&\bullet_{_5}
	}$ 
	\begin{equation}\label{mu1_A5}
\calS(\mu_1 A_5)= \begin{pmatrix}
	1& s_1 & -s_1s_2&0&0\\
	0 & 1 & -s_2&0&0 \\
	0&0&1&-s_3&0\\
	0&0&0&1&-s_4\\
	0&0&0&0&1
	\end{pmatrix}
	\end{equation}
$\mu_5A_5=\xymatrix{
	\bullet_{_1}\ar[r]&\bullet_{_2}\ar[r]&\bullet_{_3}\ar[r]&\bullet_{_4}&\bullet_{_5}\ar[l]
	}$
	\begin{equation}\label{mu5_A5}
\calS(\mu_5 A_5)=\begin{pmatrix}
	1& -s_1 & 0&0&0\\
	0 & 1 & -s_2&0&0 \\
	0&0&1&-s_3&0\\
	0&0&0&1&s_4\\
	0&0&0&0&1
	\end{pmatrix}
	\end{equation}
$\mu_3A_5=
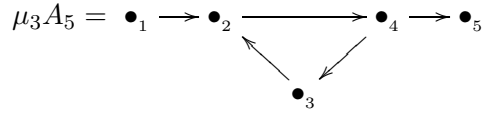
\xymatrix@=15pt{
	\bullet_{_1}\ar[r]&\bullet_{_2}\ar[rr]&&\bullet_{_4}\ar[dl]\ar[r]&\bullet_{_5}\\
	&& \bullet_{_3}\ar[ul]&
	}$
	\begin{equation}\label{mu3_A5}
\calS(\mu_3 A_5)= \begin{pmatrix}
	1 & -s_1 & 0 &0&0\\ 0 &1& -s_2s_3 & -s_2&0\\ 0 &0 & 1&0&0\\
	0& 0 & s_3 & 1&-s_4\\ 0&0&0&0&1
	\end{pmatrix}
	\end{equation}
$\mu_1\mu_4A_5 = 
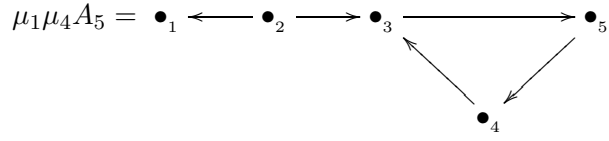
\xymatrix{
	 \bullet_{_1}  & \bullet_{_2} \ar[l]\ar[r]& \bullet_{_3}\ar[rr]	&& \bullet_{_5}\ar[dl] \\
	& & & \bullet_{_4}\ar[ul]
	}$
	\begin{equation}\label{mu14_A5}
	\calS(\mu_1\mu_4 A_5) = \begin{pmatrix}
	1 & s_1 & -s_1s_2 & 0 & 0  \\
	 0 & 1 & -s_2 & 0 & 0\\  0 & 0 & 1 &-s_3s_4 &-s_3\\  0 & 0 & 0 &1 & 0\\ 0 & 0 & 0 & s_4&1
	\end{pmatrix}
	\end{equation}
$\mu_2\mu_4A_5 = 
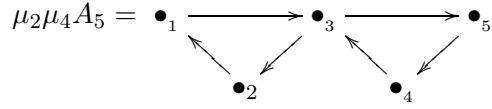
\xymatrix@=13pt{
	 \bullet_{_1}  \ar[rr]&& \bullet_{_3}\ar[dl] \ar[rr]	&& \bullet_{_5}\ar[dl] \\
	&\bullet_2 \ar[ul] & & \bullet_{_4}\ar[ul]
	}$
	\begin{equation}\label{mu24_A5}
	\calS(\mu_4\mu_2 A_5)= \begin{pmatrix}
	1 & -s_1s_2 & -s_1 &0 &0\\ 0 &1&0 & 0 &0\\ 0 &s_2 & 1&-s_3s_4 &-s_3\\
	0 &0 &0 & 1& 0\\
	0& 0 &0 & s_4 & 1
	\end{pmatrix}
	\end{equation}
$\mu_2\mu_1\mu_4A_5 = 
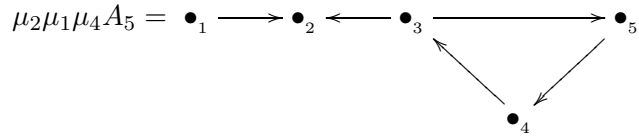
\xymatrix{
	 \bullet_{_1}  \ar[r]& \bullet_{_2} & \bullet_{_3}\ar[l]\ar[rr]	&& \bullet_{_5}\ar[dl] \\
	& & & \bullet_{_4}\ar[ul]
	}$
	\begin{equation}\label{mu214_A5}
	\calS(\mu_2\mu_1\mu_4 A_5)=\begin{pmatrix}
	1 & -s_1 &  0 & 0 & 0 \\  0 & 1 & s_2 & -s_2s_3s_4 & -s_2s_3\\  0 & 0 & 1 &-s_3s_4 & -s_3\\ 0 & 0 & 0 &1& 0 \\ 0 & 0 & 0 &s_4&1
	\end{pmatrix}
	\end{equation}
$\mu_1\mu_2\mu_1\mu_4A_5 = 
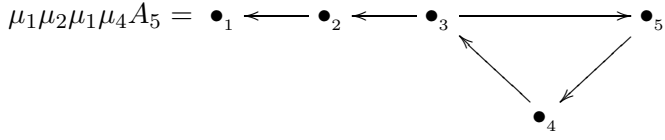
\xymatrix{
	 \bullet_{_1}  & \bullet_{_2} \ar[l]& \bullet_{_3}\ar[l]\ar[rr]	&& \bullet_{_5}\ar[dl] \\
	& & & \bullet_{_4}\ar[ul]
	}$
	\begin{equation}\label{mu1214_A5}
	\calS(\mu_1\mu_2\mu_1\mu_4 A_5)=\begin{pmatrix}
	1 & s_1 & s_1s_2 & -s_2s_3s_4 & -s_1s_2s_3s_4\\  0 & 1 & s_2 & -s_2s_3s_4 & -s_2s_3\\  0 & 0 & 1 &-s_3s_4 & -s_3\\  0 & 0 & 0 &1& 0 \\ 0 & 0 & 0 &s_4&1
	\end{pmatrix}
	\end{equation}
$\mu_1\mu_3A_5 =\xymatrix{
	\bullet_{_1}&\bullet_{_2}\ar[l]\ar[rr]&&\bullet_{_4}\ar[dl]\ar[r] & \bullet_{_5} \\
	&& \bullet_{_3}\ar[ul]& &
	}$
	\begin{equation}\label{mu13_A5}
	\calS(\mu_1\mu_3 A_5)= \begin{pmatrix}
	1 & s_1 & -s_1s_2s_3 &-s_1s_2 &0\\ 0 &1& -s_2s_3 & -s_2 &0\\ 0 &0 & 1&0 &0\\
	0 &0 &s_3 & 1& -s_4\\
	0& 0 &0 & 0 & 1
	\end{pmatrix}
	\end{equation}
$\mu_5\mu_1\mu_3A_5 = \xymatrix{
	\bullet_{_1}&\bullet_{_2}\ar[l]\ar[rr]&&\bullet_{_4}\ar[dl] & \bullet_{_5} \ar[l]\\
	&& \bullet_{_3}\ar[ul]& &
	}$
	\begin{equation}\label{mu513_A5}
	\calS(\mu_5\mu_1\mu_3 A_5) = \begin{pmatrix}
	1 & s_1 & -s_1s_2s_3 &-s_1s_2 &0\\ 0 &1& -s_2s_3 & -s_2 &0\\ 0 &0 & 1&0 &0\\
	0 &0 &s_3 & 1& s_4\\
	0& 0 &0 & 0 & 1
	\end{pmatrix}
	\end{equation}
We may now evaluate the Stokes matrices at the special point ${\bf s}_J$ and compare them.
\begin{proof}[Proof of Theorem \ref{mainThmMutations} (ii).] Write $\calS$ for $\calS_{|{\bf s}=\bf{1}}(A_n)$ and $\calS(\mu A_n)$ for $\calS_{|{\bf s}=\bf{1}}(\mu A_n)$ for brevity. 

First we observe that when $\mu$ is a simple mutation then $\calS(\mu A_n)$ and $\calS$ are actually related by the action of permutation and diagonal matrices $I_k$, or that of the braid group. Specifically:
	\begin{align*}
	\text{if }\mu=\mu_1 \text{ then } \calS(\mu A_n)&=\beta_{1,2}.\calS,\\
	\text{if }\mu=\mu_k,\ k=2,\dots,n-1, \text{ then } \calS(\mu A_n)&=P_{k,k+1}\big( \beta_{k,k+1}.\calS \big)P_{k,k+1},\\
	\text{if }\mu=\mu_n \text{ then } \calS(\mu A_n)&=I_n\calS I_n,
	\end{align*}
where $I_k$, $k=1,\dots, n$, is the matrix which differs from the identity only for the sign of the $(k,k)$ entry. This is enough to cover the case $n \leq 3$.

For $n = 4, 5$ we can use various symmetries and reduce the claim to checking a small number of cases, namely \eqref{mu1_A4}-\eqref{mu14_A4}, \eqref{mu2_A4}-\eqref{mu13_A4} and \eqref{mu1_A5}-\eqref{mu24_A5}, \eqref{mu13_A5}-\eqref{mu513_A5} above. 

In the case of $A_4$ we compute:
\begin{itemize}
\item[\eqref{mu21_A4}] $\calS(\mu_2\mu_1 A_4) = \beta_{3,4}^{-1}.\calS$,
\item[\eqref{mu121_A4}] $\calS(\mu_1\mu_2\mu_1 A_4) = \beta_{1,2}.\big(\beta_{2,3}.\big(\beta_{1,2}.\calS\big)\big)$,
\item[\eqref{mu4121_A4}] $\calS(\mu_4\mu_1\mu_2\mu_1 A_4) = I_4\calS(\mu_1\mu_2\mu_1 A_4) I_4$.
\end{itemize}
Similarly in the case of $A_5$ we find:
\begin{itemize}
\item[\eqref{mu214_A5}] $\calS(\mu_2\mu_1\mu_4 A_5) = \beta_{2,3}.\big(\beta_{1,2}.\calS(\mu_4 A_5)\big)$,
\item[\eqref{mu1214_A5}] $\calS(\mu_1\mu_2\mu_1\mu_4 A_5) = \beta_{1,2}.\calS(\mu_2\mu_1\mu_4 A_5) = I_5 \calS(\mu_1\mu_2\mu_1 A_5) I_5$.
\end{itemize}
\end{proof}

\subsection{Further examples} Let us consider the mutation-equivalent quivers of Figure \ref{introAnnulusFigure}. They come from triangulations of a surface with two boundary components, with one and two marked points respectively. For the unoriented cycle we have
\begin{equation*}
	\begin{cases}
	\alpha_1= [S_1]+\qquad [S_3] \\
	\alpha_2= \qquad - [S_2] + [S_3] \\
	\alpha_3= \qquad\qquad [S_3] 
	\end{cases} \qquad 
	\calS = \left(\begin{matrix}
1 & -s_1 s_2 & s_1 \\
 0  &  1 & 0 \\
 0  &  -s_2  & 1 
\end{matrix}\right).
	\end{equation*}
For the oriented cycle with a double arrow we have
\begin{equation*}
\begin{cases}
	\alpha_1= [S_1]+[S_2]+ [S_3] \\
	\alpha_2= \qquad [S_2] + [S_3] \\
	\alpha_3= \qquad\qquad [S_3] 
	\end{cases} \qquad 
\calS' = \left(\begin{matrix}
1 & -s_1 & -s_1 s_2 \\
 0  &  1 & s_2 \\
 0 &  0   & 1 
\end{matrix}\right).
\end{equation*}
After evaluation at the point ${\bf s}_J$, $\calS$ and $\calS'$ may be compared and we have
	\begin{equation*}
	\calS = I_1 I_2 P_{2,3} \calS' P_{2,3} I_2 I_1.
	\end{equation*}

\end{document}